\documentclass[pdflatex,sn-mathphys-num]{sn-jnl}

\usepackage{graphicx}%
\usepackage{multirow}%
\usepackage{amsmath,amssymb,amsfonts}%
\usepackage{amsthm}%
\usepackage{mathrsfs}%
\usepackage[title]{appendix}%
\usepackage{xcolor}%
\usepackage{textcomp}%
\usepackage{manyfoot}%
\usepackage{booktabs}%
\usepackage{algorithm}%
\usepackage{algorithmicx}%
\usepackage{algpseudocode}%
\usepackage{listings}%

\usepackage{bbm} 
\usepackage{mathabx}
\usepackage{mathtools}
\mathtoolsset{showonlyrefs}

\theoremstyle{thmstyleone}%
\newtheorem{theorem}{Theorem}[section]
\newtheorem{proposition}[theorem]{Proposition}%
\newtheorem{lemma}[theorem]{Lemma}
\newtheorem{condition}{}

\numberwithin{equation}{section}

\theoremstyle{thmstyletwo}%
\newtheorem{remark}{Remark}%

\theoremstyle{thmstylethree}%

\begin{document}

\title[Stable Laws and  Products of Positive Random Matrices]{Convergence to Stable Laws and a Local Limit Theorem for Products of Positive Random Matrices}


\author[1]{\fnm{Jianzhang} \sur{Mei}}\email{mjz20@mails.tsinghua.edu.cn}
\author*[2]{\fnm{Quansheng} \sur{Liu}}\email{quansheng.liu@univ-ubs.fr}
\affil[1]{ \orgname{Tsinghua University}, \orgdiv{Department of Mathematical Sciences}, \orgaddress{Beijing, 100084, China}}
\affil*[2]{ \orgname{Univ Bretagne Sud}, \orgdiv{CNRS UMR 6205, LMBA}, \orgaddress{F-56000, Vannes, France}}

%
%

%
%


\abstract{We consider the products $G_n = A_n \cdots A_1$ of independent and identical distributed nonnegative  $d \times d$ matrices $(A_i)_{i \geq 1}$.  
For any starting point $x \in  \mathbb R_+^d$ with unit norm, we establish the convergence to a stable law for the norm cocycle  $\log | G_nx |$, jointly  with its direction $G_n \cdot x = G_n x / | G_n x |$. We also prove 
a local limit theorem for the couple $ (\log |G_nx|, G_n \cdot x)$, and find the  exact rate of its convergence. 
}




\maketitle

\section{Introduction and main results}
Let $d \geq 1$ be an integer, and $(A_{i})_{i \geq 1}$ be a sequence of independent identically distributed (i.i.d.) $d \times d$ nonnegative random matrices (whose entries are all nonnegative). Define
\begin{equation}
	G_n = A_n \cdots A_1, \quad \forall n \geq 0, 
\end{equation}
with the convention that $G_0$ stands for the identity matrix. 
For a vector $x = (x_1, \cdots, x_d)^T \in \mathbb R^d$, denote its $L^1$ norm by $|x| = \sum_{i=1}^d |x_i|$. 
Let $\mathbb{R}^d_+ = \{ (x_1, \cdots, x_d)^T : x_i \geq 0 \mbox{ for } i = 1, \cdots, d   \}$ and $\mathbb{S}_+^{d-1} = \{ x \in \mathbb{R}^{d}_+ : |x| = 1 \}$. For a  nonnegative matrix $g$, 
we denote  its operator norm  $\| g \|$  and the counterpart  $\iota(g)$   as
\begin{align}  
	\| g \| & = \sup \{ |gx| : x \in \mathbb{S}_+^{d-1}  \} = \max_{j=1, \cdots, d}\sum_{i=1}^d g (i,j),  \label{def-norm-iota-1} \\
	\iota(g) & = \inf\{ |gx| : x \in \mathbb{S}_+^{d-1}  \} = \min_{j=1, \cdots, d}\sum_{i=1}^d g (i,j). \label{def-norm-iota-2}
\end{align} 
For a general  matrix $g$, we still denote by $\|g \|$ its  operator norm (with respect to the $L^1$ vector norm). 
For a matrix $g$ and a vector $x \in \mathbb{R}^d$ with $|gx| \neq 0$, define the direction of $gx$ and the norm-cocycle by 
\begin{equation}
	g\cdot x =  \frac{gx }{ |gx|} \quad \mbox{ and } 
	\quad \sigma(g, x) = \log\frac{|gx|}{|x|}.
\end{equation} 

Limit theorems for products of random matrices have been 
extensively studied
since the seminal work of Furstenberg and Kesten~\cite{Furstenberg-Kesten-1960} and Furstenberg~\cite{Furstenberg-1963}. 
For comprehensive treatments, see for example 
the books by 
Bougerol and Lacroix \cite{Bougerol-Lacroix-1985}, 
Benoist and Quint  \cite{Benoist-Quint-2016-book},  the long paper by Guivarc'h and Lepage  \cite{GuiLe16}, and the many references therein. 
Laws  of large numbers for the operator norm  $ \|G_n\|$ and  the vector norm $ | G_n x|$  with a starting point $x \in \mathbb S_+^{d-1}$,  
were  first established by  Furstenberg and Kesten~\cite{Furstenberg-Kesten-1960} and Furstenberg~\cite{Furstenberg-1963}. 
Central limit theorems were later proved by Le Page~\cite{Page1982} and  Benoist and Quint \cite{Benoist-Quint-2016-article} for invertible matrices under different moment conditions, 
and by Hennion~\cite{hennion1997limit} for nonnegative matrices. 
Large deviations  and the rate of convergence in the central limit theorems 
have been the focus of recent work by many authors,   for both invertible and nonnegative matrices:  see e.g.  Buraczewski and  Mentemeier \cite{buraczewski2016precise}, 
Sert \cite{Sert-AOP-2019},  
Xiao, Grama and Liu~\cite{XiaoGramaLiu_BerryEssen22, Xiao-Grama-Liu-AOP-2023, Xiao-Grama-Liu-JTP-2025}, 
Cuny,  Dedecker,  Merlev\`ede and  Peligrad \cite{Cuny-Dedecker-Merlevede-Peligrad-AOP-2023},
and   Cuny,  Dedecker,  Merlev\`ede  \cite{Cuny-Dedecker-Merlevede-2024}.

%
%

In this paper, we study the convergence to stable laws for products of positive random matrices. 
For the one-dimensional case ($d = 1$),  this is a classical topic;  see e.g. 
the books by 
Ibragimov and Linnik~\cite{Ibragimov-Linnik-1971} and  Petrov~\cite{Sum-Petrov}. 
For the multidimensional case ($d > 1$), this problem has been  considered by Hennion and Herv\'e~\cite{hennion2008stable}, who proved  the weak convergence to a stable law of the norm cocycle $ \log | G_n x|  $ with suitable norming.  
%
%
%
Here 
we go further by establishing 
the convergence to a stable  law of   $ \log | G_n x|$ jointly with its direction $ G_n\cdot x$, by investigating 
the weak convergence, the local limit theorem and the rate of convergence in law of the Markov chain
\begin{equation} \label{def-Snx-Xnx} 
	(S_n^x, X_n^x) 
	:= (\sigma(G_n, x), G_n\cdot x)
	=     \Big(\log | G_n x|,  \frac{G_n x}{|G_n  x|}\Big), 
\end{equation}
as $n \to \infty$, for any fixed $x \in\mathbb{S}_+^{d-1}$.

Our results extend the classical ones for random walks on the real line to the case of (non-commutative) random walks on the semigroup of nonnegative matrices.
In addition to their theoretical significance, we believe that these results provide valuable tools for applications in various research domains, such as branching random walks driven by products of random matrices and multitype branching processes in random environments.


\subsection{Convergence to stable laws and a local limit theorem}
The statements and proofs of our results are closely related to the context of Hennion and Herv\'e~\cite{hennion2008stable}. We first recall the conditions used there. For a matrix $g$, we write $g(i,j)$ for its $(i,j)$-th entry, and $g>0$ to mean that all its entries are strictly positive. 

\begin{condition}\label{cond::allowability_and_positivity}
	(Allowability and Positivity) Almost surely (a.s.), every column and row of $A_1$ contains at least one strictly positive entry, and  
	$\mathbb P [ \exists n \geq 1, G_n >0] >0.$
\end{condition}


A measurable function $L: \mathbb R_+ \to \mathbb R_+$ is called slowly varying if 
$L(t) >0$ for $t>0$ large enough and 
$\lim_{t \to +\infty} \frac{L(at)}{L(t)} = 1$ for any $a > 0$. 
By the notation \eqref{def-norm-iota-1} and \eqref{def-norm-iota-2}, 
$$  \|A_1\| = \max_{j=1, \cdots, d}\sum_{i=1}^d A_1 (i,j),  \quad
\iota (A_1) =  \min_{j=1, \cdots, d}\sum_{i=1}^d A_1 (i,j). $$

\begin{condition}\label{cond::hennion}
	There exist $\alpha \in (0, 2]$, a slowly varying function $L$ which goes to $+\infty$  if $\alpha = 2$,  and two constants $c_+ \geq 0, c_- \geq 0$ with $c_+ + c_- > 0$, such that as $t \to +\infty$,
	\begin{equation}\label{frac-t-alpha}
		\frac{t^{\alpha} \mathbb{P}[\log \|A_1\| > t]}{L(t)} = c_+ + o(1), \quad  \frac{t^{\alpha} \mathbb{P}[\log \| A_1 \| \leq -t]}{L(t)} = c_- + o(1), \quad  \frac{t^{\alpha} \mathbb{P}[\log \iota (A_1) \leq -t]}{L(t)} = O(1).
	\end{equation}
\end{condition}
Since all matrix norms are equivalent, 
the condition remains equivalent when the operator norm $\| A_1	\|$ is replaced by any matrix norm. In \cite{hennion2008stable}, the  entry-wise $L^1$-matrix norm
$\|A_1\|_{1, 1} = \sum_{i, j = 1}^d A_1(i, j)$ is used.

Denote the law of $A_1$ by $\mu$. Define the transfer operator $P$ by
\begin{equation} \label{def-P}
	Pf(x) = \mathbb{E}[f(A_1 \cdot x)] = \int f(g\cdot x)d\mu(g), \quad x \in \mathbb{S}^{d-1}_+,  
\end{equation}
for any bounded measurable function $f : \mathbb{S}^{d-1}_+ \to \mathbb{C}$.
From~\cite[Theorem 2.1]{hennion2008stable}, there exists a unique \mbox{$\mu$-stationary} probability measure $\nu$ on $\mathbb{S}^{d-1}_+$ in the sense  that for any measurable bounded $f : \mathbb{S}^{d-1}_+ \to \mathbb{C}$, 
\begin{equation}\label{nu-f-int-S-d-1-f-y}
	\nu(f) := \int_{\mathbb{S}^{d-1}_+}f(y) d\nu(y) = \int_{\mathbb{S}^{d-1}_+}Pf(y) d\nu(y) =: \mu * \nu (f). 
\end{equation} 
Let $\Gamma_\mu = [\mathrm{supp}(\mu)]$ be the closed multiplicative semigroup generated by the support of $\mu$, and $\Lambda(\Gamma_\mu)$ be the closure of $\{ v_a : a \in \Gamma_\mu, a >0 \}$, where $v_a \in \mathbb{S}_+^{d-1}$ is the Perron-Frobenius right eigenvector of $a$ with unit norm. We say that $\mu$ is non-arithmetic if for any $t > 0$, $\theta \in [0, 2\pi)$ and any function $\vartheta : \mathbb{S}^{d-1}_+ \to \mathbb{R}$, there exist $g \in \Gamma_\mu, x \in \Lambda(\Gamma_\mu)$, such that
\begin{equation}
	\exp \big( it\log|gx| -i\theta + i(\vartheta(g\cdot x) -\vartheta(x)) \big) \neq 1.
\end{equation}

\begin{theorem}[Convergence to stable laws and local limit theorem]\label{thm::local_limit_theorem}
	Assume Conditions \ref{cond::allowability_and_positivity} and \ref{cond::hennion}. Then, there exist  two sequences of real numbers $(a_n), (b_n)$, with  $a_n \geq 0$ and $\lim_{n \to \infty}a_n = \infty$, and an $\alpha$-stable law $s_\alpha$, 
	such that for any $x \in \mathbb{S}^{d-1}_+$, as $n \to \infty$,
	\begin{align} 
		\Big(\frac{S_n^x}{a_n}-b_n, X_n^x\Big) \to {s}_\alpha \otimes \nu 
		\quad  \mbox{ in law, }
	\end{align} 
	Moreover, if additionally $\alpha \neq 2$ and $\mu$ is non-arithmetic, then for any continuous function $f : \mathbb{S}_+^{d-1} \to \mathbb{R}$ and any directly Riemann integrable function $k : \mathbb{R} \to \mathbb{R}$, 
	\begin{equation}\label{lim-n-sup-x-y-S-d-1}
		\lim_{n\to\infty}\sup_{(x, y) \in \mathbb{S}_+^{d-1}\times \mathbb{R}} \bigg|a_n \mathbb{E}[f(X_n^x)k(y + S_n^x - a_nb_n)] - \nu(f) \int_\mathbb{R} k(z)p_\alpha\big(\frac{z-y}{a_n}\big)dz\bigg| = 0,
	\end{equation}
	where $p_\alpha$ is the probability density function of $s_\alpha$. 
\end{theorem}

The weak convergence of the renormalized cocycle $S_n^x/a_n - b_n$ to a stable law was proven in Hennion and Herv\'e~\cite[Theorem 1.1 and Lemma 2.1]{hennion2008stable}. 
Theorem~\ref{thm::local_limit_theorem} improves their result by establishing the convergence of the joint law 
with  the direction $X_n^x$,  and providing a local limit theorem. 
For $\alpha = 2$, the local limit theorem  was shown in Bui, Grama and Liu~\cite{bui2020asymptotic} under some exponential moment condition.  For $\alpha <2$, it is new even for 
the marginal law of  $S_n^x$.

\subsection{Exact rate of convergence}
To derive the exact rate of convergence, we need stronger conditions as follows. 
\begin{condition}[Furstenberg-Kesten condition]\label{cond::Furstenberg-Kesten}
	There exists a constant $K > 1$ such that
	\begin{equation}\label{0-leq-max}
		0 < \max_{1 \leq i,j\leq d} g(i, j) \leq K \min_{1\leq i,j \leq d} g(i, j), \quad \forall g \in \mathrm{supp}(\mu),
	\end{equation}
	where we recall that {$g(i, j)$} is the $(i,j)$-th entry of $g$ and $\mathrm{supp}(\mu)$ is the support of $\mu$. 
\end{condition}
We notice that Condition~\ref{cond::Furstenberg-Kesten} implies~\ref{cond::allowability_and_positivity}.

In the following, $X_0$ denotes a $ \mathbb S_+^{d-1}$-valued random variable whose distribution is the invariant measure $\nu$,  which is independent of $A_1$.  	
We will use the following  condition 
on the distribution of  $Z = \log | A_1 X_0 |$ 
about non-lattice and second order regular variation, introduced in de Haan and Peng~\cite{de1999exact}. 
\begin{condition}\label{cond::stable_1}
	The law of $Z = \log | A_1 X_0 |$ is non-lattice, whose distribution function $F(x) = \mathbb{P}[Z \leq x]$, $x \in \mathbb{R}$, satisfies the following properties:  there exist $\alpha \in (0, 2)$, $p\in[0, 1]$, $q \in \mathbb{R}$, $\rho \in (\alpha - 2, \alpha - 1) \cap (-\infty, 0]$ and a measurable function $A: \mathbb R_+ \to \mathbb{R}$ with $\lim_{t \to +\infty}A(t) = 0$,  
	which does not change sign for $t>0$ large enough, 
	such that
	\begin{align}\label{lim-t-frac-c-1-F-tx}
		& \lim_{t \to +\infty}\frac{\frac{1-F(tx)+F(-tx)}{1-F(t)+F(-t)} - x^{-\alpha}}{A(t)} = x^{-\alpha} \frac{x^\rho - 1}{\rho}, \; \forall x > 0, \quad \mbox{and } \; \lim_{t \to +\infty} \frac{\frac{1-F(t)}{1-F(t) + F(-t)} - p}{A(t)} = q,
	\end{align}
	with the convention that $\frac{x^0-1}{0}= \log x$ if $\rho = 0$. Moreover, $\mathbb{E}[Z] = 0$
	if  $\alpha \in (1, 2)$. 
\end{condition}
Condition~\ref{cond::stable_1} implies that $F$ is in the domain of attraction of a stable law with index $\alpha \in (0, 2)$. By de Haan and Ferreira~\cite[Theorem B.2.1 and Remark B.3.15]{haan06extreme}, we know that $|A|$ is regularly varying with index $\rho$, that is, $\lim_{t \to +\infty} \big|\frac{A(tx)}{A(t)}\big| = x^\rho$ for any $x > 0$. Since $A$ does not change sign near $+\infty$, it holds that $\lim_{t \to +\infty} \frac{A(tx)}{A(t)} = x^\rho$ for any $x > 0$. From~\cite[Proposition 1]{de1999exact}, if $\rho < 0$, the limit 
\begin{equation}\label{c}
	c:=\lim_{t\to+\infty}t^\alpha(1-F(t)+F(-t)) > 0
\end{equation}
exists;  if $\rho = 0$, we set $c = 1$.

Conditions~\ref{cond::Furstenberg-Kesten} and \ref{cond::stable_1} imply \ref{cond::hennion}. Indeed, we will see from ~\eqref{sigma-A-1-x} of Lemma \ref{lem::bound} that there exists $C > 0$ such that the distribution functions $F_1$  and $F_2$ of $\log \|A_1\|$ and $\log \iota (A_1)$  satisfy  $F_i(t - C) \leq F(t) \leq F_i(t + C)$ for all $t \in \mathbb{R}$, $i = 1, 2$. This implies that $F_1$ and $F_2$ lie in the same domain of attraction as $F$ does.

\begin{remark} 
	Under  Condition~\ref{cond::Furstenberg-Kesten} and in the case $\rho \in (-1, 0)$,
	\eqref{lim-t-frac-c-1-F-tx} holds if and only if it holds
	when $F$  
	is replaced by the distribution function of $\log \|A_1\|$: see Lemma~\ref{lem::equivalence}.
\end{remark}

Assuming Condition~\ref{cond::stable_1},
we recall some notation in~\cite{de1999exact}. Let $U$ be the generalized inverse of the function $t \in (0, +\infty) \mapsto 1/(1-F(t)+F(-t))$. Define:
\begin{align}
	a_n & = \begin{cases}\label{a_n}
		n^{1/\alpha}, \quad & \rho < 0,  \\
		U(n), \quad & \rho = 0,
	\end{cases}  \\
	b_n  & = \begin{cases}\label{b_n}
		\int_{0}^1(n(1-F(a_nx)-F(-a_nx))-c(2p-1)x^{-\alpha})dx, \quad &0 < \alpha < 1, \\
		\int_{0}^\infty n(1-F(a_nx)-F(-a_nx))\cos x \; dx, \quad &\alpha = 1, \\
		0, \quad &1 < \alpha < 2,
	\end{cases} \\
	h_\alpha(t) & = \begin{cases}
		\exp\big( -|t|^\alpha c \Gamma(1-\alpha) \big( \cos \frac{\pi\alpha}{2} - i\,\mathrm{sgn}(t)(2p-1)\sin\frac{\pi\alpha}{2}  \big) \big), \quad & \alpha \neq 1, \\
		\exp\big( -|t| c \big( \frac{\pi}{2} - i\mathrm{sgn}(t)(2p-1)\log|t| \big) \big), \quad & \alpha = 1,
	\end{cases} \quad t \in \mathbb{R},
	\label{def-h-alpha}
\end{align}
where $\mathrm{sgn}(t) = 1$ if $t \geq 0$ and $0$ if $t < 0$.  {\color{black}It is known that  if 
	$(Z_i)_{i \geq 1}$ are i.i.d. copies of $Z$, then 
	as $n \to \infty$, the sequence $(\frac{\sum_{i=1}^n Z_i}{a_n}-b_n)_{n \geq 1}$ converges in law to an $\alpha$-stable law with characteristic function $h_\alpha$ (see~\cite[Propositions 1 and 2]{de1999exact}).}

Introduce the constants  (see~\cite[3.761]{table-07}) 
\begin{align}
	& d_{a} = \int_0^\infty x^{-a} \sin x dx =  \Gamma(1-a)\sin\frac{\pi(1-a)}{2}, \quad \forall a \in (0, 2)\label{d-alpha} \\
	& z_a = \Gamma(1-a) \sin\frac{\pi(1-a)}{2}\Big( \frac{\Gamma'(1-a)}{\Gamma(1-a)} + \frac{\pi}{2}\cot\frac{\pi(1-a)}{2} \Big), \quad \forall a \in (0, 2), \notag \\
	& c_a = \frac{z_{a-1}}{a - 1} + \frac{d_{a-1}}{(a-1)^2}, \quad \forall a \in (1, 2).
\end{align}
In the following, we use the convention that $0^a \log 0 =0$ for $a>0$. 
Define, for $t \geq 0$, 
\begin{align}
	& A_\rho(t) =\begin{cases}\label{A-rho}
		\frac{c}{\rho} d_{\alpha - \rho} t^{\alpha - \rho}, \quad & \rho < 0, \\
		t^\alpha(z_\alpha - d_\alpha \log t) , \quad & \rho = 0,
	\end{cases} 
\end{align} 
\begin{align}
	& B_\rho(t) = \begin{cases}
		\big( \frac{2p-1}{\rho} + 2q \big) \frac{cd_{\alpha - \rho - 1}}{\alpha - \rho - 1} t^{\alpha - \rho},  & 1 < \alpha < 2, \; \alpha - 2 < \rho < 0, \\
	t^\alpha \big( (2p-1) \big( c_\alpha - \frac{d_{\alpha - 1}}{\alpha - 1}\log t \big) + \frac{2qd_{\alpha - 1}}{\alpha - 1} \big), & 1 < \alpha < 2,\; \rho = 0, \\
		 \big( \frac{2p-1}{\rho} + 2q \big)  \frac{cd_{-\rho}}{-\rho} (t^{1-\rho} - t),  & \alpha = 1, \; -1 < \rho < 0, \\
		 \big( \frac{2p-1}{\rho} + 2q\big) \frac{c}{\alpha - \rho - 1}(d_{\alpha - \rho - 1}  t^{\alpha - \rho} - t)  , & 0 < \alpha < 1, \; \alpha - 2 < \rho < \alpha - 1.
	\end{cases}
	 \label{B-rho}
\end{align}
For $t<0$, we define  $A_\rho(t) = A_\rho(-t)$, 
$ B_\rho(t) = - B_\rho(-t)$. Then $\forall t \in \mathbb R$,  $A_\rho(t) = A_\rho(|t|)$, 
$B_\rho(t) = \mathrm{sgn}(t) {B}_\rho(|t|)$.


\medskip 
The following condition depicts the tail behavior of $A_1$. 
\begin{condition}\label{cond::when_modulus_large}
	There exists a measure $\tilde{\mu}$ on the space of nonnegative matrices 
	such that,   
	as $n \to \infty$,	
	the conditional laws 	$\mathbb {P} \big( \frac{A_1}{\|A_1\|}  \in \cdot \big|  \log \|A_1\| > n\big)$ and 
	$\mathbb {P} \big( \frac{A_1}{\|A_1\|}  \in \cdot \big| \log \|A_1\| \leq -n)$ converge (weakly)   to $\tilde{\mu}$. 
\end{condition}
Let $Q$ be the operator defined as follows: for any bounded measurable $f : \mathbb{S}_+^{d-1} \to \mathbb{C}$ and $x \in \mathbb{S}_+^{d-1}$,
\begin{equation}\label{def-Q}
	Qf(x) = \int f(g\cdot x)d\tilde{\mu}(g).
\end{equation}
Define
\begin{equation}
	\Delta  := \lim_{n\to\infty}\sum_{i=0}^{n-1}P^{n-1-i}(Q-P)P^{i}; 
\end{equation}
	in Lemma~\ref{lem::operator-U} we will see that 
	the limit exists in the space $\mathcal B (\mathcal L) $  of bounded linear 
	operators (equipped with the operator norm) on some Banach space $\mathcal L$,     such that 
	$  \Delta f =  \delta (f)  \mathbf 1$, where  $\mathbf 1$  denotes the constant  function on  
	$\mathbb{S}_+^{d-1}$ with value $1$, and $\delta $ is a bounded linear mapping from $\mathcal L$ to $\mathbb{C}$.   
	A series  representation of $\Delta$ and $\delta$ will also be given in that lemma. 

Let $H_\alpha$ be the distribution function whose characteristic function is $h_\alpha$ defined in \eqref{def-h-alpha}. 	
Let $J(t)=A_\rho(t)+iB_\rho(t)$ if $\rho > -\alpha$, and $J(t) = \frac{(\log h_\alpha(t))^2}{2}$ if $\rho < -\alpha$. 
	Define for $ s \in \mathbb{R}$, 
	\begin{equation} \label{def-MsNs} 
		M(s) = \frac{1}{2\pi} \int_\mathbb{R} \frac{e^{-its}}{it}J(t)h_\alpha(t)dt,
		\quad
		N(s)=\frac{1}{2\pi}\int_\mathbb{R}\frac{e^{-its}}{-it}(\log h_\alpha(t)) h_\alpha(t)dt.  
	\end{equation}



\begin{theorem}[Exact rate of convergence  in law for $(S_n^x, X_n^x)$ with suitable norming]\label{thm::main}
	Assume Conditions~\ref{cond::Furstenberg-Kesten} and \ref{cond::stable_1}
	with $\rho \neq  -\alpha$. Let $f : \mathbb{S}_+^{d-1} \to \mathbb{C}$ be a Lipschitz function with respect to the Euclidean distance.
	\begin{enumerate}[{1.}]
		\item 
		If $\rho > -\alpha$, then uniformly for $s \in \mathbb{R}$ and $x \in \mathbb{S}^{d-1}_+$,
		\begin{equation}\label{lim-n-infty-A-a-n--1}
			\lim_{n\to\infty}(A(a_n))^{-1}\bigg(\mathbb{E}\bigg[f(X_n^x)\mathbbm{1}_{\{ \frac{S_n^x}{a_n}-b_n \leq s  \}}\bigg] - \nu(f) H_\alpha(s)\bigg) = \nu(f) M(s).
		\end{equation}
		\item 
		If $\rho <-\alpha$ and Condition~\ref{cond::when_modulus_large} holds, then uniformly for $s \in \mathbb{R}$ and $x \in \mathbb{S}^{d-1}_+$,
		\begin{equation}\label{lim-n-bigg-mathbb-E}
			\lim_{n\to\infty}n\bigg(\mathbb{E}\bigg[f(X_n^x)\mathbbm{1}_{\{ \frac{S_n^x}{a_n}-b_n \leq s\}}\bigg] - \nu(f) H_\alpha(s)\bigg) = \nu(f)M(s) 
			+ \delta (f) \,  N(s), 
		\end{equation} 
		where $\delta $ is a bounded linear mapping  on the set of Lipschitz functions on $\mathbb S_+^{d-1}$ such that 
		$\Delta  f =  \delta (f)  \mathbf 1$. 
	\end{enumerate}
\end{theorem}
In the one-dimensional case $d = 1$,  Theorem~\ref{thm::main} has been proven in~\cite{de1999exact}.  Here we focus on the multidimensional case $d >1$.  Notice that by letting $f = \mathbf{1}$ and using  $\Delta \mathbf{1} = 0$,  from Theorem~\ref{thm::main} we derive  that uniformly in $s \in \mathbb{R}$ and $ x \in \mathbb{S}_+^{d-1}$,  
\begin{equation}
	\lim_{n \to \infty}l_n^{-1} \bigg(\mathbb{P}\bigg[\frac{S_n^x}{a_n} - b_n \leq s\bigg] - H_\alpha(s)\bigg) = M(s),
\end{equation}
where $l_n = A(a_n)$ if $\rho > -\alpha$ and $l_n = n^{-1}$ if $\rho < -\alpha$.
In order to prove Theorem~\ref{thm::main}, we will make use of the one-dimensional result derived in~\cite{de1999exact} together with the spectral gap theory developed in~\cite{hennion2008stable}. 

Theorem \ref{thm::main} excludes the case $\rho = -\alpha$, consistent with the one-dimensional result presented in \cite[Theorem 3 and Remark 2]{de1999exact}.

\section{
	The transfer operator $P_t$  and the proof of Theorem~\ref{thm::local_limit_theorem}}\label{sec::2}

Throughout this section, we assume Conditions~\ref{cond::allowability_and_positivity} and \ref{cond::hennion}. 
The law of  the couple $(S_n^x, X_n^x)$ defined in \eqref{def-Snx-Xnx}  can be determined  by the 
family of  transfer operators  $(P_t)_{t \in \mathbb R}$ defined as follows: for any bounded measurable function $f :   \mathbb{S}^{d-1}_+ \to \mathbb C$, 
\begin{equation}
	P_tf(x) = \mathbb{E}[e^{itS_1^x} f(X_1^x)] = \int e^{it\sigma(g, x)}f(g\cdot x)d\mu(g), \quad x \in \mathbb{S}^{d-1}_+. 
\end{equation}
Notice that $P_0 = P$. The $n$-fold composition of $P_t$  is given by
\begin{equation}
	P_t^nf(x ) = \mathbb{E}[e^{itS_n^x} f(X_n^x)], \quad x \in \mathbb{S}_+^{d-1}, \quad n \geq 1. 
\end{equation}
The following variant of Hilbert's distance $\mathbf{d}$, used in~\cite{hennion1997limit, hennion2008stable}, is important for our analysis. For $x, y \in \mathbb{S}_+^{d-1}$, define $m(x, y) = \min\{y_i^{-1}x_i : i=1, \cdots, d, y_i > 0\}$ and $\mathbf{d}(x, y) = \frac{1-m(x, y)m(y, x)}{1 + m(x, y)m(y, x)}$. This distance satisfies:
\begin{itemize}
	\item $\sup \{ \mathbf{d}(x, y) : x, y \in \mathbb{S}_+^{d-1} \} = 1$;
	\item $| x - y | \leq 2\mathbf{d}(x, y)$ for $x, y \in \mathbb{S}^{d-1}_+$; 
	\item $\mathbf{d}(g \cdot x, g\cdot y) \leq c(g) \mathbf{d}(x, y)$ for any nonnegative matrix $g$, $x, y \in \mathbb{S}^{d-1}_+$; 
	\item $c(g) \leq 1$ for any nonnegative matrix $g$, and $c(g) < 1$ if entries of $g$ are positive;
	\item $c(gg') \leq c(g)c(g')$ for any two nonnegative matrices $g, g'$.
\end{itemize}
We recall the Banach space $\mathcal{L}$ of  $\bf d$-Lipschitz functions   defined in~\cite{hennion2008stable}. Denote 
\begin{equation}
	m(f) = \sup \bigg\{ \frac{|f(x_1) - f(x_2)|}{\mathbf{d}(x_1, x_2)} : x_1, x_2 \in \mathbb{S}_+^{d-1}, x_1\neq x_2 \bigg\}
\end{equation}
for any function $f : \mathbb{S}_+^{d-1} \to \mathbb{C}$. Let $\mathcal{L} = \{f: \mathbb{S}_+^{d-1} \to \mathbb{C},   \mbox{measurable and } m(f) <+\infty \}$ be the Banach  space equipped with the norm \begin{equation}
	\| f\|_\mathcal{L} = \|f\|_{\infty} + m(f).
\end{equation}
Since $|x-y| \leq 2\mathbf{d}(x, y)$ for every $x, y\in\mathbb{S}_+^{d-1}$, the space $\mathcal{L}$ contains all Lipschitz functions on $\mathbb{S}_+^{d-1}$ with respect to the Euclidean distance.
The space ${\mathcal B} ( \mathcal L) $ of bounded linear operators  on $\mathcal L$, equipped with the  
operator norm 
still denoted by $\| \cdot \|_{\mathcal{L}}$, is also a  Banach space.  Let $\Pi$ be the rank-one projection:
\begin{align} \label{def-Pi}
	\Pi f = \nu(f) \mathbf{1}, \quad \forall f \in \mathcal L, 
\end{align} 
where $\mathbf{1}$ denotes the constant function on $\mathbb{S}_+^{d-1}$ with value one. 

For a nonnegative matrix $g$, denote $\ell(g) = |\log\|g\| | + |\log \iota(g) |$. 


We collect some useful results of~\cite{hennion2008stable} in the following proposition.  Recall that $Z= \log  | A_1 X_0 |$, where $X_0$ is independent of $A_1$ and has law $\nu$. 

\begin{proposition}[{\cite{hennion2008stable}}]\label{prop::useful}
	Assume Conditions~\ref{cond::allowability_and_positivity} and \ref{cond::hennion}. 
	\begin{enumerate}[{1.}]
		\item (Regularity of $P_t$ at $0$) As $t \to 0$, $\| P_t - P \|_{\mathcal{L}} =O (\epsilon(t) + |t|)$,
		where $\epsilon(t) = \int \min(|t|\ell(g), 2)d\mu(g)$. In particular, $\| P_t - P \|_\mathcal{L} = O(t^\beta)$ where $\beta = 1$ if $1 < \alpha \leq 2$, and $\beta < \alpha$ can be arbitrary close to $\alpha$ if $0 < \alpha \leq 1$. 
		\item (Spectral gap) There exists an interval $I$ that contains $0$, such that for each $t \in I$, $P_t$ has a unique dominant eigenvalue $\lambda(t) \in \mathbb{C}$ (i.e. the eigenvalue with largest modulus) and a rank-one corresponding eigenprojection $\Pi_t$, satisfying the following properties: 
		there exist $\kappa \in (0, 1)$, $C > 0$ such that: 
		\begin{itemize}
			\item $\lambda(0) = 1$, $\lambda$ is continuous on $I$, 
			$\kappa < |\lambda(t)| \leq 1$ for $t \in I$; 
			\item $\Pi_0 = \Pi$,    $\|\Pi_t - \Pi\|_\mathcal{L} = O( \|P_t - P\|_\mathcal{L})$ as $t \to 0$; 
			\item $\forall t \in I$, $R_t := P_t - \lambda(t)\Pi_t$  (so that $R_0= P-\Pi$) satisfies
			$ \forall n \geq 1$, 
			\begin{equation} \label{R-t-n-R-0-n}
				P_t^n = \lambda(t)^n \Pi_t + R_t^n, 
				\quad \|R_t^n\|_\mathcal{L} \leq C\kappa^n,  	
				\quad \|R_t^n - R_0^n\|_\mathcal{L} \leq C\kappa^n \| P_t - P \|_\mathcal{L}. 
			\end{equation}
		\end{itemize}
		\item (Estimation of $\lambda(t)$) As $t \to 0$, we have $\lambda(t) = \phi_Z(t) + O(\|P_t - P\|_{\mathcal{L}}^2)$, where $\phi_Z(t) = \mathbb{E}[e^{itZ}]$.
		\item (Domain of attraction) The random variable $Z$ belongs to the domain of attraction of an $\alpha$-stable law with $0 < \alpha \leq 2$, that is, there exist 
		sequences of real numbers 
		$(a_n)_{n \geq 1}, (b_n)_{n \geq 1}$ with $\lim_{n\to\infty}a_n=+\infty$, such that $\frac{\sum_{i=1}^n Z_i}{a_n} - b_n$ converges in law to an $\alpha$-stable law $s_\alpha$, where $(Z_i)_{i \geq 1}$ are i.i.d. copies of $Z$. Moreover, the sequence $(a_n)_{n\geq 1}$ can be chosen such that $\frac{nL(a_n)}{a_n^\alpha} = 1$ (with $L$ introduced in \ref{cond::hennion}).
	\end{enumerate}
\end{proposition}
\begin{proof}
	We only need to prove the property that $|\lambda(t)| \leq 1$ for $t \in I$ and the last assertion in ~\eqref{R-t-n-R-0-n}, because other results have been shown 
		in~\cite[Theorems 3.2 and 3.3, Propositions 3.1, 4.1 and 4.2, and Proof of Proposition 4.2]{hennion2008stable}.
	
		For $t \in I$, let $v_t = \Pi_t \mathbf{1}$ be an eigenfunction of $P_t$: $P_t v_t = \lambda(t) v_t$. Then, by the definition of $P_t$, we know that $|\lambda(t)| \|v_t\|_\infty = \|P_tv_t\|_\infty \leq \|v_t\|_\infty$. Thus $|\lambda(t)| \leq 1$ for $t \in I$. 
	
	Let $\kappa_1 = \frac{1+\kappa}{2}$, $D = \{ z \in \mathbb{C} : |z| = \kappa_1 \}$. For any $t \in I$, since the spectrum of $R_t$ lies inside $D$, we have $R_t^n = P_t^n - \lambda(t)^n\Pi_{t} = \frac{1}{2\pi i}\int_{\partial D} z^n (z-P_t)^{-1}dz$ for $n \geq 1$. Using the identity  $(a-b)^{-1} - a^{-1} = a^{-1} \sum_{m \geq 1} (ba^{-1})^m$ with $a = z-P_0, b = P_t-P_0$, we know there exists $C > 0$ such that $\| R_t^n - R_0^n \|_\mathcal{L} \leq C\kappa_1^n \|P_t - P\|_\mathcal{L}$ for all $n \geq 1$ and all $t \in \mathbb R$ with $|t|$ small enough, say $|t| \leq \eta$, so the last assertion in ~\eqref{R-t-n-R-0-n} holds with $\kappa$ replaced by $\kappa_1$ and $I = [-\eta, \eta]$. 
\end{proof}

Throughout this section,  
the interval $I$ and the sequences 
$(a_n)_{n\geq 1}$ and $ (b_n)_{n \geq 1}$ are as given in 
Proposition~\ref{prop::useful},  with $\frac{nL(a_n)}{a_n^\alpha} = 1$.

\begin{lemma}\label{lem::exponential-decay}
	Assume Conditions~\ref{cond::allowability_and_positivity}, \ref{cond::hennion}, and the measure $\mu$ is non-arithmetic. Then, the spectral radius of $P_t$ for $t \in \mathbb{R}\setminus \{0 \}$ is strictly smaller than one (so that $|\lambda(t)| < 1$ for $t \in I\setminus \{ 0 \}$).  Moreover, for any compact set $K \subset \mathbb{R} \setminus \{ 0 \}$ and $f \in \mathcal{L}$, there exists $r \in (0, 1)$ such that
	\begin{equation}
		\sup_{t \in K} \| P_t^n f\|_\infty \leq r^n \|f\|_\mathcal{L}, \quad \forall n \geq 1.
	\end{equation}
\end{lemma}
\begin{proof}
	By arguing as in the proof of~\cite[Propositions 3.6, 3.7 and 3.10]{XiaoGramaLiu_BerryEssen22} with $s = 0$ and $ \gamma = 1$, we can show the quasi-compactness of $P_t$ and use the non-arithmeticity to prove this lemma. 
	To mimic the proof of~\cite{XiaoGramaLiu_BerryEssen22}, the key step is to verify the conditions of the theorem of Ionescu-Tulcea
	and Marinescu (see \cite[Proposition 3.6]{XiaoGramaLiu_BerryEssen22}, \cite{Ionescu-Tulcea-Marinescu-1950, Hennion-Herve-2001}), among which we only need to verify the so-called Doeblin-Fortet inequality: for $n \geq 1$, $t \in \mathbb{R}$, there exist $C > 0, r \in (0, 1)$, such that
	\begin{equation}\label{m-P-t-n-f-C-1-t}
		m(P_t^nf) \leq C(1+|t|)\|f\|_\infty + Cr^n m(f), \quad \forall f \in \mathcal{L}.
	\end{equation}
	
	Let $\mu^{(n)}$ be the probability measure of $G_n$. For $x, y \in \mathbb{S}_+^{d-1}$, we write $\frac{P_t^nf(x) - P_t^nf(y)}{\mathbf{d}(x,y)} = I_1 + I_2$, with 
	\begin{align}
		I_1 := \int \frac{e^{it\sigma(g, x)} - e^{it\sigma(g, y)}}{\mathbf{d}(x, y)} f(g\cdot x) d\mu^{(n)}(g), \quad I_2 := \int e^{it\sigma(g, y)} \frac{f(g\cdot x) - f(g\cdot y)}{\mathbf{d}(x, y)} d\mu^{(n)}(g).
	\end{align}
	For $I_1$, since $|\sigma(g, x)-\sigma(g, y)| \leq 2|\log(1-\mathbf{d}(x, y))|$ for any nonnegative matrix  $g$ with at least one positive entry in each row and column (\cite[Lemma 3.1]{hennion2008stable}), there exists $C_1 > 0$ such that $|I_1| \leq C_1|t|\|f\|_\infty$ if $\mathbf{d}(x, y) < \frac{1}{2}$, and $|I_1| \leq 4\|f\|_\infty$ otherwise. For $I_2$, we have $|I_2| \leq m(f) c(\mu^{(n)})$, where
	\begin{equation}
		c(\mu^{(n)}) := \sup\Big\{ \int \frac{\mathbf{d}(g\cdot x_1, g\cdot x_2)}{\mathbf{d}(x_1, x_2)} d\mu^{(n)}(g) : x_1, x_2 \in \mathbb{S}_+^{d-1}, x_1\neq x_2 \Big\}.
	\end{equation}
	By the properties of the distance $\mathbf{d}$ and Condition~\ref{cond::allowability_and_positivity}, we know that the sequence $ (c(\mu^{(n)}))$ is submultiplicative and $\lim_{n}c(\mu^{(n)})^{\frac{1}{n}} < 1$. So we finish the proof of~\eqref{m-P-t-n-f-C-1-t}.
\end{proof}	

\begin{lemma}\label{lem::decay_1}
	Assume Conditions~\ref{cond::allowability_and_positivity} and \ref{cond::hennion}. Let $\beta = 1$ if $1 < \alpha \leq 2$, and $\beta \in (\frac{\alpha}{2}, \alpha)$ otherwise. Then, for any fixed $t \in \mathbb{R}$,
	\begin{equation}\label{lim-lambda-phi-fixed-t}
		\lim_{n \to \infty}\Big(\lambda\big(\frac{t}{a_n}\big)^n-\phi_Z\big(\frac{t}{a_n}\big)^n \Big) = 0.
	\end{equation}
	Moreover, if $\alpha \neq 2$, then for any $\gamma \in (0, \frac{2\beta - \alpha}{\alpha})$ and $ \varepsilon > 0$, there exist positive numbers $N, \tau, C_1, C_2$ with $[-\tau, \tau] \subset I$, such that for all $n \geq N$ and $t \in [-\tau a_n, \tau a_n]$,
	\begin{equation}\label{lambda-t-a-n-n-K-t}
		\Big|\lambda\big( \frac{t}{a_n} \big)\Big|^n \leq K(t), 
		\quad 
		\Big|\phi_Z\big( \frac{t}{a_n}\big) \Big|^n \leq K(t),
	\end{equation}
	\begin{equation}\label{lambda-t-a-n-n-phi-Z}
		\Big|\lambda\big(\frac{t}{a_n}\big)^n - \phi_Z\big( \frac{t}{a_n} \big)^n\Big| \leq C_1 K(t) |t|^{2\beta} n^{-\gamma},
	\end{equation}
	with 
	$K(t) := \exp(-C_2 |t|^{\alpha}\min(|t|^\varepsilon, |t|^{-\varepsilon}))$. 
\end{lemma}
\begin{proof}
	(1) We first  prove~\eqref{lim-lambda-phi-fixed-t}. By Part 1 of Proposition~\ref{prop::useful} and the relation $\frac{nL(a_n)}{a_n^\alpha} = 1$, there exists $c_1 > 0$ such that
	\begin{equation}\label{n-P-s-a-n-P-2-L-leq-C-1}
		n\| P_{\frac{t}{a_n}} - P\|^2_{\mathcal{L}} \leq c_1n|t|^{2\beta}a_n^{-2\beta} = c_1|t|^{2\beta} n^{\frac{\alpha - 2\beta}{\alpha}} L(a_n)^{\frac{-2\beta}{\alpha}}, \quad \forall t \in \mathbb{R}, \quad \forall n \geq 1.
	\end{equation}
	Since $L$ is unbounded if $\alpha = 2$ by Condition~\ref{cond::hennion}, we know from~\eqref{n-P-s-a-n-P-2-L-leq-C-1} that for fixed $t \in \mathbb{R}$, 
	\begin{equation}\label{lim-P-t-a-n-P-2-2-L}
		\lim_{n \to \infty} n\| P_{\frac{t}{a_n}} - P\|^2_{\mathcal{L}} = 0
	\end{equation}
	
		Notice that for $a, b \in \mathbb{C} \setminus \{ 0 \}$ and $n \geq 1$, we have 
		\begin{equation}\label{a-n-b-n}
			|a^n - b^n| \leq \max(|a|^{n-1}, |b|^{n-1}) n|a-b|.
		\end{equation}
		Indeed, without loss of generality, we assume that $z := \frac{a}{b}$ satisfies $|z| \leq 1$. Then, we have that \mbox{$|1-z^n| \leq \sum_{i=1}^n |z^{i-1}-z^{i}| \leq n|1-z|$}. 
		
		Using~\eqref{a-n-b-n} with $a = \lambda(\frac{t}{a_n})$ and $b = \phi_Z(\frac{t}{a_n})$ together with the property that \mbox{$\lambda(t) = \phi_Z(t) + O(\|P_t - P\|_\mathcal{L}^2)$} by Part 3 of Proposition~\ref{prop::useful}, we know that there exists $c_2 > 0$ such that for $t \in \mathbb{R}$ and $n \geq 1$,
		\begin{align}\label{lambda-t-a-22}
			\Big|\lambda\big(\frac{t}{a_n}\big)^n - \phi_Z\big(\frac{t}{a_n}\big)^n\Big| & \leq \max\Big( \big|\phi_Z\big( \frac{t}{a_n} \big)\big|^{n-1}, \big|\lambda\big( \frac{t}{a_n} \big)\big|^{n-1} \Big)\,n\Big|\lambda\big(\frac{t}{a_n}\big) - \phi_Z\big(\frac{t}{a_n}\big)\Big|  \notag \\
			&  \leq c_2 \max\Big( \big|\phi_Z\big( \frac{t}{a_n} \big)\big|^{n-1}, \big|\lambda\big( \frac{t}{a_n} \big)\big|^{n-1} \Big)\,n\|P_{\frac{t}{a_n}} - P\|_\mathcal{L}^2. 
		\end{align}
		Since $\big|\phi_Z\big( \frac{t}{a_n} \big)\big| \leq 1$ and $\big|\lambda\big( \frac{t}{a_n} \big)\big| \leq 1$ by Part 2 of Proposition~\ref{prop::useful}, we get from~\eqref{lim-P-t-a-n-P-2-2-L} that for fixed $t \in \mathbb{R}$, as $n \to \infty$, 
		\begin{equation}
			\Big|\lambda\big(\frac{t}{a_n}\big)^n - \phi_Z\big(\frac{t}{a_n}\big)^n\Big| \leq c_2\, n\| P_{\frac{t}{a_n}} - P\|^2_{\mathcal{L}} \to 0.
		\end{equation}
	
	(2) We next prove~\eqref{lambda-t-a-n-n-K-t}. Assume $\alpha \neq 2$. Recall that $F$ is the distribution function of $Z$. By~\cite[Proposition 4.1]{hennion2008stable} and Condition~\ref{cond::hennion}, 
	\begin{equation}\label{1-F-u-L-u-c-+}
		1-F(u) = \frac{L(u)(c_+ + o(1))}{u^\alpha}, \quad F(-u) = \frac{L(u)(c_- + o(1))}{u^\alpha}, \quad u \to +\infty.
	\end{equation}
	Since $Z$ is in the domain of attraction of an $\alpha$-stable law, from~\cite[Theorem 2.6.5]{Ibragimov-Linnik-1971} we know that $\log \phi_Z(t)$ has the form 
	\begin{equation}\label{log-phi-Z-t}
		\log \phi_Z(t) = i\gamma t - c_3|t|^\alpha \tilde{L}\big( \frac{1}{|t|} \big) \big( 1 - i\chi \,\mathrm{sgn}(t) w(t, \alpha) \big), \quad t \in \mathbb{R} \setminus \{0\},
	\end{equation}
	where $\gamma \in \mathbb{R}, c_3 > 0$, $\chi \in [-1,1]$, $\tilde{L}$ is a slowly varying function, $w(t, \alpha) = \tan(\frac{\pi \alpha}{2})$ if $\alpha \neq 1$ and $w(t, 1) = \frac{2\log|t|}{\pi}$. From the proof of~\cite[Theorem 2.6.5]{Ibragimov-Linnik-1971} in the case $0 < \alpha < 2$, we can deduce from~\eqref{1-F-u-L-u-c-+} that $\tilde{L}(s) = (c_4 + o(1))L(s)$ as $s \to +\infty$ for some $c_4 > 0$. In particular, with $c_5 = c_3c_4$, we have $|\phi_Z(t)| = \exp(-c_5|t|^\alpha L(\frac{1}{|t|})(1 + o(1)))$ as $t \to 0$. Since $\lambda(t) = \phi_Z(t) + O(\|P_t - P\|_\mathcal{L}^2) = \phi_Z(t) + O(|t|^{2\beta})$ and $2\beta > \alpha$, we can choose $\tau_1 > 0$ small enough and $c_6 \in (0, c_5)$ such that $|\lambda(t)| \leq \exp(-c_6|t|^\alpha L(\frac{1}{|t|}))$ when $|t| \leq \tau_1$, hence
	\begin{equation}\label{lambda-t-a-n-n}
		\Big|\lambda\big( \frac{t}{a_n}\big)\Big|^n \leq \exp\bigg(-c_6|t|^\alpha \frac{nL(a_n)}{a_n^\alpha} \frac{L(\frac{a_n}{|t|})}{L(a_n)}\bigg), \quad \forall t \in [-\tau_1 a_n, \tau_1 a_n]. 
	\end{equation}
	Since $L$ is slowly varying, by Karamata's characterization theorem (\cite[Theorem B.1.6]{haan06extreme}), there exist two measurable functions $b: \mathbb R_+ \to \mathbb R$ and $ c:  \mathbb R_+ \to \mathbb R_+$ with $\lim_{u \to \infty}b(u) = 0$ and $\lim_{u \to\infty} c(u) = 1$, such that $L(u) = c(u)\exp(\int_{u_0}^u \frac{b(x)}{x} dx)$ for $u \geq 1$. It follows that for any $\varepsilon > 0$, we can choose $\tau_2 > 0$ small enough and $N \geq 0$ such that
	\begin{equation}\label{L-a-n-t-L-a-n}
		\frac{L(\frac{a_n}{|t|})}{L(a_n)} = \frac{c(\frac{a_n}{|t|})}{c(a_n)} \exp\bigg( \int_{a_n}^{\frac{a_n}{|t|}} \frac{b(u)}{u} du \bigg) \geq \frac{1}{2}\min(|t|^{\varepsilon}, |t|^{-\varepsilon}), \quad \forall t \in [-\tau_2 a_n, \tau_2 a_n], \quad \forall n \geq N
	\end{equation}
	(to see~\eqref{L-a-n-t-L-a-n}, we can discuss two cases $|t| \leq 1$ and $1 \leq |t| \leq \tau_2a_n$). Fix $\varepsilon > 0$ and choose $\tau \leq \min(\tau_1, \tau_2)$ small enough such that  $[-\tau, \tau] \subset I$. Combining~\eqref{lambda-t-a-n-n} ,~\eqref{L-a-n-t-L-a-n} and the condition $\frac{nL(a_n)}{a_n^\alpha} = 1$, we have 
	\begin{equation}\label{lambda-t-a-n-n-leq-exp}
		\Big|\lambda(\frac{t}{a_n})\Big|^n \leq \exp\Big(-\frac{c_6}{2} |t|^{\alpha}\min(|t|^\varepsilon, |t|^{-\varepsilon})\Big), \quad \forall t \in [-\tau a_n, \tau a_n], \quad \forall n \geq N. 
	\end{equation}
	This proves the first inequality in~\eqref{lambda-t-a-n-n-K-t}. The proof for the second inequality is similar. 
	
		(3) We then prove~\eqref{lambda-t-a-n-n-phi-Z}. 
		Since $\gamma < \frac{2\beta-\alpha}{\alpha}$ and $L$ is slowly varying, we have $n^{\frac{\alpha - 2\beta}{\alpha}} L(a_n)^{\frac{-2\beta}{\alpha}} = o(n^{-\gamma})$ as $n \to \infty$. Plugging \eqref{lambda-t-a-n-n-K-t} and \eqref{n-P-s-a-n-P-2-L-leq-C-1} into \eqref{lambda-t-a-22}, we get \eqref{lambda-t-a-n-n-phi-Z}.
\end{proof}

We now come to prove Theorem~\ref{thm::local_limit_theorem}. 
\begin{proof}[Proof of Theorem \ref{thm::local_limit_theorem}]

	(1) We first prove the weak convergence of the couple $(\frac{S_n^x}{a_n}-b_n, X_n^x)$. Let $x \in \mathbb{S}_+^{d-1}$, $f \in \mathcal{L}$ and $t \in \mathbb{R}$. 
	Denote by $h_\alpha$ the characteristic function of $s_\alpha$. By Part 4 of Proposition~\ref{prop::useful}, the characteristic function $\phi_Z$ of $Z$ satisfies $\lim_{n\to\infty} e^{-itb_n}\phi_Z(\frac{t}{a_n})^n \to h_\alpha(t)$. Notice that 
	\begin{equation}\label{E-e-it-S-n-x-a-n-b-n}
		\mathbb{E}[e^{it(S_n^x/a_n-b_n)} f(X_n^x)] = e^{-itb_n}P_{\frac{t}{a_n}}^nf(x) = e^{-itb_n}\Big(\lambda\big(\frac{t}{a_n}\big)^n \Pi_{\frac{t}{a_n}}f(x) + R_{\frac{t}{a_n}}^nf(x)\Big).
	\end{equation}
	Recall that $\Pi f = \nu(f) \mathbf{1}$. From Parts 1 and 2 of Proposition~\ref{prop::useful}, we see that there exists $C_2 > 0$ such that \mbox{$\|\Pi_{\frac{t}{a_n}}f - \Pi f\|_\infty \leq C_2 \|P_{\frac{t}{a_n}}-P\|_\mathcal{L}\to 0$}, and $R^n_{\frac{t}{a_n}} f(x) \to 0$ as $n \to \infty$. By \eqref{lim-lambda-phi-fixed-t} and \eqref{E-e-it-S-n-x-a-n-b-n}, we have
	\begin{equation}\label{lim-E-e-it}
		\lim_{n\to \infty}\mathbb{E}[e^{it(S_n^x/a_n-b_n)} f(X_n^x)] = \lim_{n \to \infty} \Pi f(x) e^{-itb_n}  \phi_Z\big(\frac{t}{a_n}\big)^n = \nu(f)h_\alpha(t).
	\end{equation}
	For each $v \in \mathbb{R}^d$, define $g_v(y) = e^{i\langle v, y\rangle}$, $y \in \mathbb{R}^d$, where $\langle v, y\rangle  := \sum_{i=1}^dv_iy_i$ is the scalar product of $v = (v_1, \cdots, v_d)^T \in \mathbb{R}_+^d$ and $y = (y_1, \cdots, y_d)^T \in \mathbb{R}_+^d$. Letting $f = g_v$ in~\eqref{lim-E-e-it}, we have
	\begin{equation}
		\lim_{n\to \infty}\mathbb{E}[e^{it(S_n^x/a_n-b_n)+i\langle v, X_n^x \rangle}] = h_\alpha(t)\int e^{i\langle v, y \rangle}d\nu(y), \quad \forall v \in \mathbb{R}^d.
	\end{equation}
	Since $t \in \mathbb{R}$ and $v \in \mathbb{R}^{d}$ are arbitrary, by L\'evy's continuity theorem, this proves that $(\frac{S_n^x}{a_n}-b_n, X_n^x)$ converges in law to $s_\alpha \otimes \nu$. 
	
	(2) We then prove the local limit theorem. Our proof follows the approach of~\cite{bui2020asymptotic}. Let $f \in \mathcal{L}$ and $k \in L^1(\mathbb{R})$ be such that the Fourier transform 
	\begin{equation}
		\hat{k}(t) := \int_{\mathbb{R}}e^{-iut}k(u)du
	\end{equation}
	has support within $[-l, l]$ for some $l > 0$. By the Fourier inversion formula, for any $(x, y )\in \mathbb{S}_+^{d-1} \times \mathbb{R}$,
	\begin{align}
		a_n \mathbb{E}[f(X_n^x)k(y + S_n^x -a_nb_n)] & = \frac{a_n}{2\pi} \mathbb{E}\Big[f(X_n^x) \int_\mathbb{R} e^{ity + itS_n^x-ita_nb_n}\hat{k}(t)dt\Big] \notag \\
		& = \frac{1}{2\pi} \int_\mathbb{R} e^{it(\frac{y}{a_n} - b_n)} \hat{k}\big(\frac{t}{a_n}\big) P_{\frac{t}{a_n}}^n f(x) dt.
	\end{align}
	Let $M > 1, \tau \in (0, l)$ small enough with $[-\tau, \tau] \subset I$. We write 
	\begin{equation}
		a_n \mathbb{E}[f(X_n^x)k(y + S_n^x -a_nb_n)] = \sum_{i=1}^8 I_i,
	\end{equation}
	where
	\begin{align}
		I_1 & = \frac{1}{2\pi} \int_{\tau a_n \leq |t| \leq la_n} e^{it(\frac{y}{a_n} - b_n)} \hat{k}\big(\frac{t}{a_n}\big) P_{\frac{t}{a_n}}^n f(x) dt, \notag \\ 
		I_2 & =\frac{1}{2\pi} \int_{|t| \leq \tau a_n} e^{it(\frac{y}{a_n} - b_n)} \hat{k}\big(\frac{t}{a_n}\big) R_{\frac{t}{a_n}}^n f(x) dt, \notag \\
		I_3 & = \frac{1}{2\pi}\int_{M \leq |t| \leq \tau a_n} e^{it(\frac{y}{a_n} - b_n)} \hat{k}\big( \frac{t}{a_n} \big) \lambda\big( \frac{t}{a_n} \big)^n \Pi_{\frac{t}{a_n}}f(x) dt, \notag \\
		I_4 & =\frac{1}{2\pi} \int_{|t| \leq M} e^{it(\frac{y}{a_n} - b_n)} \hat{k}\big( \frac{t}{a_n} \big) \Big(\lambda\big( \frac{t}{a_n} \big)^n - \phi_Z\big(\frac{t}{a_n}\big)^n  \Big)\Pi_{\frac{t}{a_n}}f(x)dt   \notag \\
		I_5 & = \frac{1}{2\pi} \int_{|t| \leq M} e^{it(\frac{y}{a_n} - b_n)} \hat{k}\big( \frac{t}{a_n} \big) \phi_Z\big( \frac{t}{a_n}\big)^n(\Pi_{\frac{t}{a_n}} f(x) - \nu(f))dt, \notag \\
		I_6 & = \frac{\nu(f)}{2\pi} \int_{|t| \leq M} e^{\frac{ity}{a_n}} \hat{k}\big( \frac{t}{a_n} \big) \Big(e^{-itb_n}\phi_Z\big( \frac{t}{a_n} \big)^n - h_\alpha(t)\Big) dt, 
		\nonumber \\ 
		I_7 & = \frac{-\nu(f)}{2\pi} \int_{|t| \geq M} e^{\frac{ity}{a_n}} \hat{k}\big( \frac{t}{a_n} \big) h_\alpha(t) dt, \notag \\
		I_8 & = \frac{\nu(f)}{2\pi} \int_\mathbb{R} e^{\frac{ity}{a_n}} \hat{k}\big( \frac{t}{a_n}\big) h_\alpha(t)dt.
	\end{align}
	We have, uniformly in $(x, y)$, $\lim_{n \to \infty}I_1 = 0$ by Lemma~\ref{lem::exponential-decay}, $\lim_{n \to \infty}I_2 \to 0$ by Part 2 of Proposition~\ref{prop::useful}, $\lim_{n \to \infty}I_4 = \lim_{n \to \infty} I_5 = \lim_{n \to \infty} I_6 = 0$ for any fixed $M$ by the dominated convergence theorem and Lemma~\ref{lem::decay_1}. We also get that, uniformly in $(n, x, y)$,  $\lim_{M \to \infty} I_3 = 0$ and $\lim_{M \to \infty} I_7 = 0$, again by dominated convergence theorem and Lemma~\ref{lem::decay_1}. Thus, we get that $\lim_{n \to \infty} | a_n \mathbb{E}[f(X_n^x)k(y + S_n^x -a_nb_n)] - I_8 | = 0$ uniformly in $(x, y)$. 
	
	Denote the inverse Fourier transform of a function $g \in L^1(\mathbb{R})$  by $\widecheck{g}(u) := \frac{1}{2\pi} \int_\mathbb{R} e^{iut}g(t)dt$, $u \in \mathbb R$. By the definition of $h_\alpha$, we have $h_\alpha(t) = \int_\mathbb{R} e^{iut}p_\alpha(u)du, t \in \mathbb{R}$, which implies that $h_\alpha = 2\pi\widecheck{p_\alpha}$. Let $\psi = \frac{h_\alpha}{2\pi} = \widecheck{p_\alpha}$ (so $\hat{\psi} = p_\alpha$ by the Fourier inversion theorem). Note that $\hat{\varphi}(t) = e^{\frac{ity}{a_n}}\hat{k}(\frac{t}{a_n})$ is the Fourier transform of the function $\varphi(u) := a_n k(a_n u + y)$.
	Using Parseval's identity $\int_{\mathbb{R}} \hat{\varphi}(t) \psi(t) dt = \int_{\mathbb{R}} \varphi(t) \hat{\psi}(t) dt$, we know that
	\begin{equation}
		I_8 = \frac{\nu(f)}{2\pi} \int_\mathbb{R} e^{\frac{ity}{a_n}} \hat{k}\big( \frac{t}{a_n}\big) h_\alpha(t)dt = \nu(f) \int_\mathbb{R} a_n k(a_nu + y)p_\alpha(u) du = \nu(f) \int k(u)p_\alpha\big( \frac{u-y}{a_n} \big) du. 
	\end{equation}
	It follows that~\eqref{lim-n-sup-x-y-S-d-1} holds for $f\in \mathcal{L}$ and $k\in L^1(\mathbb{R})$ such that $\hat{k}$ has compact support. We can then argue as in~\cite[Proof of Theorem 2.2]{bui2020asymptotic} to establish~\eqref{lim-n-sup-x-y-S-d-1} for any continuous function $f$ and directly Riemann integrable function $k$. 
\end{proof}

\section{The proof of Theorem~\ref{thm::main}}\label{sec::3}
In this section,  we  study the convergence rate in law of $(S_n^x, X_n^x)$, 
assuming Conditions~\ref{cond::Furstenberg-Kesten} and \ref{cond::stable_1} in place of Conditions~\ref{cond::allowability_and_positivity} and \ref{cond::hennion}. Since Conditions \ref{cond::Furstenberg-Kesten} and \ref{cond::stable_1} imply \ref{cond::allowability_and_positivity} and \ref{cond::hennion}, Proposition \ref{prop::useful} and Lemma  
\ref{lem::decay_1}
still apply. 

\subsection{Auxiliary   lemmas}  We first give some auxiliary results required for the proof of Theorem~\ref{thm::main}.

The  
first  lemma concerns a property of the cocycle $\sigma (A_1, x) = \log |A_1 x| $, the contraction of the action of $A_1$ on $\mathbb S_+^{d-1}$, and an  improvement of 
Proposition \ref{prop::useful}  about  the regularity of $P_t$ at $0$ for $\alpha <1$. 

\begin{lemma}\label{lem::bound}
	Assume Condition~\ref{cond::Furstenberg-Kesten} and \ref{cond::stable_1}. Then, there exist $C > 0$, $r \in (0, 1)$ such that 
	a.s. 
	for any $x, y \in \mathbb{S}_+^{d-1}$, 
	\begin{equation}\label{sigma-A-1-x}
		|\sigma(A_1, x) - \log \|A_1\|| \leq C, \quad \mathbf{d}(A_1\cdot x, A_1\cdot y) \leq r \, \mathbf{d}(x, y). 
	\end{equation} 
	If additionally   $\alpha < 1$, then
	\begin{equation}\label{P-t-P-L-O}
		\|P_t-P\|_\mathcal{L} = O(|t|^\alpha), \quad \mbox{ as } t\to 0.
	\end{equation}
\end{lemma}
\begin{proof}
	The first  inequality  in~\eqref{sigma-A-1-x} follows from~\cite[Lemma 5.1] {mei2024support}, while the second is a consequence of  Condition~\ref{cond::Furstenberg-Kesten} and \cite[Lemma 10.7]{hennion1997limit}. 
	
	When $\alpha < 1$, we have $\rho < \alpha - 1 < 0$, so  Condition \ref{cond::stable_1} implies that the limit $c$ defined in~\eqref{c} exists with $c>0$ (see \cite[Proposition 1]{de1999exact}). 
	Using the first inequality in~\eqref{sigma-A-1-x}, we know that the function $\epsilon(t)$ in Part 1 of Proposition~\ref{prop::useful} satisfies $\epsilon(t) \leq \mathbb{E}[\min(|t|(2|Z|+2C), 2)]$ for $t \in \mathbb{R}$. Since
	\begin{equation}
		\mathbb{P}[|Z| > s] \leq (1-F(s)+F(-s)) \leq \frac{2c}{s^{\alpha}}
	\end{equation}
	for $s > 0 $ large enough by~\eqref{c}, a simple computation shows that $\epsilon(t) = O(|t|^\alpha)$, hence we have $\mathbb\| P_t - P \|_\mathcal{L} = O(\epsilon(t) + |t|) = O(|t|^\alpha)$ as $t \to 0$.
\end{proof}

	The second lemma gives an equivalent version  of the condition  \eqref{lim-t-frac-c-1-F-tx} when $\rho \in (-1, 0)$. 
	\begin{lemma}\label{lem::equivalence}
		Assume Condition~\ref{cond::Furstenberg-Kesten}, $\alpha \in (0, 2)$, $p \in [0, 1]$, $q \in \mathbb{R}$, $\rho \in (-1, 0)$, and $A: \mathbb{R}_+ \to \mathbb{R}$ is a measurable function such that $\lim_{t\to+\infty}A(t) = 0$ and that does not change sign for $t > 0$ large enough. Then, \eqref{lim-t-frac-c-1-F-tx} holds if and only if it holds when $F$ is replaced by the distribution function of $\log \|A_1\|$.
	\end{lemma}
	\begin{proof}
		We only prove the necessity, because the proof for the sufficiency is similar. Suppose that~\eqref{lim-t-frac-c-1-F-tx} holds. Let $g(t) = t^\alpha (1-F(t)+F(-t)), t > 0$. Since $\rho < 0$, we know that the limit $c = \lim_{t \to +\infty} g(t)$ defined in~\eqref{c} exists with $c > 0$. From~\eqref{lim-t-frac-c-1-F-tx}, we get
		\begin{equation}
			\lim_{t \to +\infty}\frac{g(tx)-g(t)}{cA(t)} = \lim_{t \to +\infty}\frac{g(tx)-g(t)}{g(t)A(t)}=\frac{x^\rho-1}{\rho}, \quad \forall x > 0.
		\end{equation}
		By~\cite[Theorem B.2.2]{haan06extreme}, we know that 
		\begin{equation}\label{lim-t-c-g-t}
			\lim_{t \to +\infty} \frac{c-g(t)}{cA(t)}=\frac{-1}{\rho}.
		\end{equation}
		
		Let $F_1(x) = \mathbb{P}[\log \|A_1\| \leq x]$, $x \in \mathbb{R}$, be the distribution function of $\log \|A_1\|$, and 
		$$ g_1(t)=t^\alpha(1-F_1(t)+F_1(-t)),  \quad  t > 0.$$
		From Lemma~\ref{lem::bound}, there exists $C > 0$ such that $F(t-C) \leq F_1(t) \leq F(t+C)$ for all $t \in \mathbb{R}$. 
		Thus for all $t > C$, 
		\begin{equation}\label{1+C-t}
			\Big(1+\frac{C}{t}\Big)^{-\alpha} g(t+C) \leq g_1(t) \leq \Big(1-\frac{C}{t}\Big)^{-\alpha} g(t-C). 
		\end{equation}
		Notice that $|A|$ is $\rho$-regularly varying (see \cite[Theorem B.2.1 and Remark B.3.15]{haan06extreme}). Since $\rho > -1$, we know that $(1\pm\frac{C}{t})^{-\alpha}-1=O(\frac{1}{t})=o(A(t))$ as $t \to +\infty$. Thus, from~\eqref{lim-t-c-g-t} and \eqref{1+C-t}, we get
		\begin{equation}
			\lim_{t \to +\infty} \frac{c-g_1(t)}{cA(t)}=\frac{-1}{\rho}.
		\end{equation}
		This implies that 
		\begin{equation}
			\lim_{t \to +\infty} \frac{g_1(tx)-g_1(t)}{cA(t)} = \lim_{t \to \infty} \Big(\frac{c-g_1(t)}{cA(t)} - \frac{c-g_1(tx)}{cA(tx)} \frac{A(tx)}{A(t)}\Big) = \frac{x^\rho - 1}{\rho}, \quad \forall x > 0,
		\end{equation}
		hence the first assertion in~\eqref{lim-t-frac-c-1-F-tx} holds when $F$ is replaced by $F_1$. 
		The proof for the second assertion in~\eqref{lim-t-frac-c-1-F-tx} with $F$ replaced by $F_1$ is similar. 
	\end{proof}

The third lemma is a technical result 
stated without proof in~\cite[Proof of Theorem 1]{de1999exact}. 
As it plays an important  role in our analysis, we provide here a sketch of proof  for completeness.
\begin{lemma}[{\cite{de1999exact}}]\label{lem:esseen-estimate}
	Assume Conditions~\ref{cond::Furstenberg-Kesten} and \ref{cond::stable_1}. Let $A_\rho$ and $B_\rho$ be defined as in~\eqref{A-rho} and \eqref{B-rho}. 
	\begin{enumerate}[{1.}]
		\item If $\rho > -\alpha$, then for any $\eta \in (0, \frac{1}{\alpha - \rho})$,
		we have, with $l_n = A(a_n)$ and 
		$m_n = |l_n|^{-\eta}$, 
		\begin{equation}\label{lim-n-1-l-n}
			\lim_{n \to \infty} \frac{1}{l_n} \int_{|t| \leq m_n} \frac{1}{|t|} \bigg|e^{-itb_n}\phi_Z\big(\frac{t}{a_n}\big)^n-h_\alpha(t) + l_n h_\alpha(t) (A_\rho(t) + iB_\rho(t)) \bigg| dt = 0.
		\end{equation}
		
		\item If $\rho < -\alpha$, then for any $\eta \in (0, \frac{1}{2\alpha})$, we have,  with $m_n = n^{\eta}$, 
		\begin{equation}\label{lim-n-t-leq-m-n}
			\lim_{n \to \infty} n \int_{|t| \leq m_n} \frac{1}{|t|} \bigg|e^{-itb_n}\phi_Z\big(\frac{t}{a_n}\big)^n-h_\alpha(t) + n^{-1} h_\alpha(t) \frac{(\log h_\alpha(t))^2}{2} \bigg|dt = 0.
		\end{equation}
	\end{enumerate}
\end{lemma}
\begin{proof}
	We only sketch the proof for the case where $1 < \alpha < 2$ and $\rho \in (\alpha-2, 0)$ (so that $\rho > \alpha - 2 > -\alpha$); the other cases can be proved similarly by using~\cite[Lemmas 4, 5 and 6]{de1999exact}. 
	For simplicity, denote $l = l_n = A(a_n)$, $a = A_\rho(t)$, $b = B_\rho(t)$, $f_1 = e^{-itb_n}\phi_Z(\frac{t}{a_n})^n$, and $f_2 = h_\alpha(t)$. Let $\varepsilon > 0$ be small enough, and write 
	$$d_1 = d_1(t, \varepsilon) = (1-\varepsilon)\min(|t|^\varepsilon, |t|^{-\varepsilon}),  \quad d_2 = d_2(t, \varepsilon) = (1 + \varepsilon)\max(|t|^\varepsilon, |t|^{-\varepsilon}).$$ 
	Our argument is based on 
	~\cite[Lemma 4 (i)]{de1999exact}, whose  assertion  depends on the sign of 
	$2p-1+2qp$.  
	For simplicity, we only consider the case where  $2p-1+2qp \geq 0$;  the opposite  case can be treated similarly.
	In this case,  \cite[Lemma 4 (i)]{de1999exact} implies that  there exists $N_0 > 0$ such that for  all $n \geq N_0$ and $t \in [-n^{\frac{1}{\alpha}}/N_0, n^{\frac{1}{\alpha}}/N_0]$,  
	\begin{equation}
		d_1\leq \mathrm{Re}((\log f_2 - \log f_1)/(la)) \leq d_2, \quad d_1 \leq \mathrm{Im}((\log f_2 - \log f_1)/(lb)) \leq d_2.
	\end{equation}
	Since $d_1 \leq 1 \leq d_2$, we see that $|\mathrm{Re}(\log f_2 - \log f_1) - la| \leq |l|a \max(d_2 - 1, 1 - d_1) \leq |l|a(d_2-d_1)$;  similarly $|\mathrm{Im}(\log f_2 - \log f_1) - lb| \leq |lb|(d_2 - d_1)$. Therefore, when $n$ is large enough, for all  $t \in [-m_n, m_n]$,
	\begin{equation}\label{log-f-2-log-f-1-l}
		|\log f_2 - \log f_1 - l(a + ib)| \leq |l|(a+|b|)(d_2 - d_1).
	\end{equation}
	Recall that $a = \frac{c}{\rho}d_{\alpha-\rho}|t|^{\alpha - \rho}$ and $ b = \mathrm{sgn}(t)\big( \frac{2p-1}{\rho} + 2q \big) \frac{cd_{\alpha - \rho - 1}}{\alpha - \rho - 1} |t|^{\alpha - \rho}$. From~\eqref{log-f-2-log-f-1-l}, we see that there is a constant 
	$C_0>0$ such that for $n$ large enough, 
	\begin{equation}\label{max-t-leq-m-n-log}
		\max_{|t| \leq m_n}|\log f_1 - \log f_2| \leq \max_{|t| \leq m_n} |l|(a+|b|)(d_2-d_1+1) \leq C_0 |l|m_n^{\alpha-\rho + \varepsilon} 
		= C_0 |l|^{1-\eta(\alpha-\rho) - \eta \varepsilon}. 
	\end{equation}
	Since $\eta \in (0, \frac{1}{\alpha - \rho})$, we have $1 -\eta (\alpha - \rho) > 0$. Taking $\varepsilon >0$  small enough, we get 
	\begin{equation}\label{lim-n-t-t-t}
		\lim_{n\to\infty} \max_{|t| \leq m_n}|\log f_1 - \log f_2| = 0.
	\end{equation}  
	Using the Taylor expansion $e^x - 1 - x = O(|x|^2)$ with $x = \log f_1 - \log f_2$, from~\eqref{lim-n-t-t-t} we get that for some constant $C > 0$ and all $n $ large enough,  
	\begin{equation}\label{f-1-f-2-f-2-log}
		|f_1 - f_2 - f_2(\log f_1 - \log f_2)| \leq C|f_2||\log f_1 - \log f_2|^2 \quad \mbox{ if } |t| \leq m_n. 
	\end{equation}
	Combining~\eqref{log-f-2-log-f-1-l}, \eqref{max-t-leq-m-n-log} and \eqref{f-1-f-2-f-2-log}, we have that for all $n$ large enough and $t \in [-m_n, m_n]$,
	\begin{align}
		|f_1 - f_2 + lf_2(a+ib)| & \leq 
		|f_1 - f_2 - f_2(\log f_1 - \log f_2)| + |f_2||\log f_1 - \log f_2 + l(a + ib)| \notag \\
		& \leq C|f_2||\log f_1 - \log f_2|^2 + |f_2||l|(a+|b|)(d_2 - d_1) \notag \\
		& \leq C  |f_2||l|^2(a+|b|)^2(d_2-d_1 + 1)^2 + |f_2||l|(a+|b|)(d_2-d_1).
	\end{align}
	Thus, the integral in~\eqref{lim-n-1-l-n} is bounded by $o(|l|) + C|l| \int_{\mathbb{R}} |f_2|(d_2-d_1) (a+|b|)/|t| dt$. 
	Passing to the limit as $n\to \infty $  and then as  $\varepsilon \to 0$, and using the fact that 
	$\lim_{\varepsilon \to 0} (d_2-d_1) = 0$ for all $t \in \mathbb{R}$, we get ~\eqref{lim-n-1-l-n}. 
\end{proof}



The fourth lemma 
is a version of Esseen's   smoothing inequality. 
The difference with the usual version is that  here we have the perturbation  term $G$ on the difference of two bounded non-decreasing functions $F_1$ and $F_2$. 

\begin{lemma}[Esseen-type inequality]\label{lem::Esseen}
	Let $F_1, F_2 : \mathbb{R} \to \mathbb{R}$ be two bounded non-decreasing functions such that $\lim_{x \to \pm \infty} (F_1(x) - F_2(x)) = 0$, $f_i(t) := \int_\mathbb{R} e^{itx}dF_i(x)$, $t \in \mathbb{R}$, $i = 1, 2$. Define for $x \in \mathbb{R}$, $G(x) = \frac{1}{2\pi} \int_\mathbb{R} e^{-itx} \frac{g(t)}{-it} dt$, where $g : \mathbb{R} \to \mathbb{C}$   is measurable   such that  its complex conjugate  $\overline g $ satisfies $ \overline g(t) = g(-t) $ for all $t \in \mathbb R \setminus \{0\}$, and that  
	both $g$ and  $t \mapsto \frac{g(t)}{t}$ are in $ L^1(\mathbb{R})$. If $F_2$ is differentiable on $\mathbb{R}$ whose derivative satisfies $\|F_2'\|_\infty < \infty$, and $t \mapsto \frac{f_1(t) - f_2(t) - g(t)}{t}$ is in $L^\infty ([a,b])$ 
	(the space of essentially bounded functions on $[a,b]$)
	for all $a, b \in \mathbb R$ with $a<b$, then for any $T > 0$, 
	\begin{equation}\label{sup-x-R-F-1-F-2-G-1}
		\sup_{x \in \mathbb{R}}|F_1(x) - F_2(x) - G(x)| \leq \frac{24\|F_2' + G'\|_\infty}{\pi T} + \frac{1}{\pi} \int_{-T}^T \bigg| \frac{f_1(t) - f_2(t) - g(t)}{t} \bigg|dt.
	\end{equation}
\end{lemma}
\begin{proof}
	Fix $T > 0$. Let $V_T$ be the probability distribution with density $v_T(x) = \frac{1 - \cos(Tx)}{\pi T x^2}, x \in \mathbb{R} \setminus \{0\}$, $v_T(0) =  \frac{T}{2 \pi}$, and $w_T(t) = \int_\mathbb{R} e^{itx} dV_T(x) = \max\{ 0, 1-\frac{|t|}{T} \}$, $t \in \mathbb{R}$. For a bounded function $K : \mathbb{R} \to \mathbb{R}$, we consider the convolution $K * V_T(x) = \int_\mathbb{R} K(x-y)v_T(y)dy$, $x \in \mathbb{R}$. 
	
	Since $g \in L^1(\mathbb{R})$, we have $G'(x) = \frac{1}{2\pi} \int_\mathbb{R} e^{-itx}g(t)dt$ $\forall x \in \mathbb{R}$, hence $\|G'\|_\infty < \infty$. 
	Notice that  $G$ is  real-valued since $\overline g(t) = g(-t) $. 
	Using~\cite[Lemma XVI.3.1]{Feller-1971-Vol-2} for the non-decreasing function $F_1$ and the function $F_2 + G$ with $\|F_2' + G'\|_\infty < \infty$, we know that
	\begin{equation}
		\sup_{x \in \mathbb{R}} |F_1(x) - F_2(x) - G(x)| \leq \frac{24\|F_2' + G'\|_\infty}{\pi T} +  2 \sup_{x \in \mathbb{R}}|F_1 * V_T(x) - F_2 * V_T(x) - G * V_T(x)|.
	\end{equation}
	So, in order to prove~\eqref{sup-x-R-F-1-F-2-G-1}, it suffices to prove
	\begin{equation}\label{sup-F-1-V-T-x}
		\sup_{x \in \mathbb{R}}|F_1 * V_T(x) - F_2 * V_T(x) - G * V_T(x)| \leq \frac{1}{2\pi} \int_{-T}^T \bigg| \frac{f_1(t) - f_2(t) - g(t)}{t} \bigg|dt.
	\end{equation}
	
	Note that $\int_{\mathbb{R}}e^{itx} d(F_i * V_T)(x) = f_i(t)w_T(t)$ for $t \in \mathbb{R}$, $i = 1, 2$. Since $f_i w_T \in L^1(\mathbb{R})$, from the Fourier inversion theorem, we know that $F_i * V_T$ is differentiable, and
	\begin{equation}
		(F_i * V_T)'(u) = \frac{1}{2\pi} \int_\mathbb{R} e^{-itu} f_i(t)w_T(t)dt, \quad  \forall u \in \mathbb{R}, \quad i = 1, 2.
	\end{equation}
	By integrating this identity and using Fubini's theorem (and the fact that $w(t) = 0 $ when $|t| >T$), 
	we know that for any $a, x \in \mathbb{R}$, we have
	\begin{equation}\label{F_i-1}
		(F_i * V_T)(x) - (F_i * V_T)(a) = \frac{1}{2\pi} \int_{\mathbb{R}} \frac{e^{-itx} - e^{-ita}}{-it} f_i(t)w_T(t) dt, \quad i = 1, 2.
	\end{equation}
	
	We notice that by the definition of $G$, for $x \in \mathbb{R}$,
	\begin{equation}
		G * V_T(x) = \int_\mathbb{R} G(x-y)V_T(y)dy = \frac{1}{2\pi} \int_\mathbb{R}  \bigg( \int_\mathbb{R} e^{-it(x-y)}\frac{g(t)}{-it} v_T(y) dt  \bigg) dy. 
	\end{equation}
	Using Fubini's theorem and the condition that the function $t\mapsto \frac{g(t)}{t} $ is in $L^1(\mathbb{R})$, we have for $x \in \mathbb{R}$,
	\begin{equation}
		G * V_T(x) = \frac{1}{2\pi} \int_\mathbb{R} e^{-itx} \frac{g(t)w_T(t)}{-it} dt. 
	\end{equation}
	It follows that for any $a, x \in \mathbb{R}$, 
	\begin{equation}\label{G-x}
		(G * V_T)(x) - (G * V_T)(a) = \frac{1}{2\pi} \int_\mathbb{R} \frac{e^{-itx} - e^{-ita}}{-it} g(t)w_T(t) dt. 
	\end{equation}
	Combining \eqref{F_i-1} and \eqref{G-x}, we know that $H := (F_1 - F_2 - G) * V_T$ satisfies for any $a, x \in \mathbb{R}$, 
	\begin{equation}\label{Delta}
		H (x) - H (a) = \frac{1}{2\pi} \int_{\mathbb{R}} \frac{e^{-itx} - e^{-ita}}{-it} (f_1(t) - f_2(t) - g(t))w_T(t)dt. 
	\end{equation}
	On the one hand, we have $\lim_{a \to \pm \infty} H (a) = 0$ since $\lim_{x \to \pm \infty} (F_1(x) - F_2(x)) = \lim_{x \to \pm \infty} G(x) = 0$ (by the condition and  the Riemann-Lebesgue lemma). On the other hand, since the function $t \mapsto \frac{f_1(t) - f_2(t) - g(t)}{t}$ is in  $L^\infty ( [-T, T])$, we have  $\lim_{a \to \pm \infty}\int_{\mathbb R} e^{-ita} \frac{f_1(t) - f_2(t) - g(t)}{t}w_T(t)dt = 0$ again by the Riemann-Lebesgue lemma. Thus, taking $a \to - \infty$ in ~\eqref{Delta} we get that
	\begin{equation}
		H (x) = \frac{1}{2\pi} \int_\mathbb{R} 
		e^{-itx}  \frac{f_1(t) - f_2(t) - g(t)}{-it} w_T(t) dt.  
	\end{equation}
	Taking absolute value and supremum on $x\in\mathbb{R}$, we derive~\eqref{sup-F-1-V-T-x}. 
\end{proof}

\subsection{Proof of Theorem~\ref{thm::main} for the case  $\rho > -\alpha$}  

We can now give the proof of 
Theorem~\ref{thm::main} for  the case $\rho > -\alpha$.

\begin{proof}[Proof of Theorem~\ref{thm::main}, case $\rho > -\alpha$]
	Let  $f \in \mathcal{L}$. 
	We assume that $f$ is real-valued and $\min_{y \in \mathbb{S}_+^{d-1}}f(y) > 0$ without loss of generality; for the general case we can use the decomposition $f = \mathrm{Re}(f) + i \,\mathrm{Im}(f)$  when $f$ is complex-valued, and 
	$f = (f_+ +1) - (f_- +1)$ with	  $f_+ = \max (f, 0) $, $f_- = \max (-f; 0)$, when $f$ is real-valued. 
	As in Lemma \ref{lem:esseen-estimate},  we write $ l_n = A(a_n)$.
	Recall that $\Pi f = \nu(f)\mathbf{1}$ and that 
	$M(s) = \frac{1}{2\pi} \int_\mathbb{R} \frac{e^{-its}}{it}h_\alpha(t)J(t)dt$ for $s \in \mathbb{R}$ (see  \eqref{def-MsNs}).
	
		To prove~\eqref{lim-n-infty-A-a-n--1}, we need to show that 
		\begin{equation}\label{lim-l-n-sup=x-1}
			\lim_{n \to \infty} l_n^{-1} \sup_{x \in \mathbb{S}_+^{d-1}, s \in \mathbb{R}}\Big|\mathbb{E}[f(X_n^x) \mathbbm{1}_{\{\frac{S_n^x}{a_n}-b_n \leq s\}}] - \nu(f)H_\alpha(s) - \nu(f)l_nM(s)\Big| = 0. 
		\end{equation}
		Since  $ \sup_{x \in \mathbb{S}_+^{d-1}}  | \mathbb{E}[f(X_n^x)] - \nu (f)  | 
		=  \sup_{x \in \mathbb{S}_+^{d-1}} | (\Pi + R_0^n)f(x) - \nu(f) | = \|R_0^n\|_\infty \to 0$ exponentially fast as $n \to \infty$, it suffices to show that 
		\begin{equation}\label{lim-l-n-sup=x}
			\lim_{n \to \infty} l_n^{-1} \sup_{x \in \mathbb{S}_+^{d-1}, s \in \mathbb{R}}\Big|\mathbb{E}[f(X_n^x) \mathbbm{1}_{\{\frac{S_n^x}{a_n}-b_n \leq s\}}] - \mathbb{E}[f(X_n^x)]H_\alpha(s) - \nu(f)l_nM(s)\Big| = 0. 
		\end{equation} 

		We choose $\beta > 0$ such that $\beta = 1$ if $\alpha \in (1, 2)$, and $\beta \in (\frac{\alpha-\rho}{2}, \alpha)$ if $\alpha \leq 1$. Note that  $\beta = 1 > \frac{\alpha-\rho}{2}$ when $\alpha \in (1, 2)$, since $\rho > \alpha-2 = \alpha-2\beta$. Choose $\gamma \in (\frac{-\rho}{\alpha}, \frac{2\beta-\alpha}{\alpha})$. With these choices, from Lemma~\ref{lem::decay_1} we know that there exist positive numbers $N, \tau, C_1, C_2$ such that \eqref{lambda-t-a-n-n-K-t} and  \eqref{lambda-t-a-n-n-phi-Z} 
		hold for $n \geq N$ and $t \in [-\tau a_n, \tau a_n]$, with $K(t) = \exp(-C_2 |t|^{\alpha}\min(|t|^{\frac{\alpha}{2}}, |t|^{-\frac{\alpha}{2}}))$. Define $T = \tau a_n$. 
	Using Lemma~\ref{lem::Esseen} with $F_1(s) = \mathbb{E}[f(X_n^x)\mathbbm{1}_{\{\frac{S_n^x}{a_n}-b_n \leq s\}}]$, $F_2(s) = \mathbb{E}[f(X_n^x)]H_\alpha(s)$ and $G(s) = \nu(f)l_nM(s)$, we get that for all $x \in \mathbb{S}_+^{d-1}$ and $n \geq N$,
	with $C_n := \mathbb{E}[f(X_n^x)]\|H_\alpha'\|_\infty + l_n\nu(f) \|M'\|_\infty$, 
	\begin{align}\label{l-n--1-sup-s-in-R}
		& l_n^{-1} \sup_{s \in \mathbb{R}}\Big|\mathbb{E}[f(X_n^x) \mathbbm{1}_{\{\frac{S_n^x}{a_n}-b_n \leq s\}}] - \mathbb{E}[f(X_n^x)]H_\alpha(s) - \nu(f)l_nM(s)\Big| \notag \\
		& \leq  \frac{24C_n }{\pi l_n T} +  \frac{1}{\pi} \int_{-T}^T \bigg|\frac{e^{-itb_n}P_{\frac{t}{a_n}}^n f(x) - P_0^nf(x)h_\alpha(t) + \nu(f)l_nh_\alpha(t)J(t)}{l_nt}\bigg|dt.
	\end{align}
	From \eqref{R-t-n-R-0-n},  we have, for
	all $x \in \mathbb{S}_+^{d-1}$ and
	$t \in I$,
	\begin{align}
		& e^{-itb_n}P_{\frac{t}{a_n}}^nf(x) - P_0^nf(x)h_\alpha(t) + \nu(f) l_n h_\alpha(t)J(t) \notag \\
		= \; &  e^{-itb_n}\lambda\big(\frac{t}{a_n}\big)^n (\Pi_{\frac{t}{a_n}}f(x) - \nu(f)) + (e^{-itb_n}R_{\frac{t}{a_n}}^n - h_\alpha(t)R_0^n)f(x)   \notag \\
		& + \nu(f)e^{-itb_n}\Big(\lambda\big(\frac{t}{a_n} \big)^n - \phi_Z\big(\frac{t}{a_n}\big)^n \Big)
		+ \nu(f)\Big(e^{-itb_n}\phi_Z\big(\frac{t}{a_n}\big)^n - h_\alpha(t) + l_nh_\alpha(t)J(t)\Big).\label{e-itb-n-p}
	\end{align}
	Plugging this into~\eqref{l-n--1-sup-s-in-R}, we get that for all $x \in \mathbb{S}_+^{d-1}$ and $n \geq N$,
	\begin{align}\label{l-n--1-sup-s-in-R-2}
		& l_n^{-1} \sup_{s \in \mathbb{R}}\Big|\mathbb{E}[f(X_n^x) \mathbbm{1}_{\{\frac{S_n^x}{a_n}-b_n \leq s\}}] - \mathbb{E}[f(X_n^x)]H_\alpha(s) - \nu(f)l_nM(s)\Big| \notag \\
		& \leq 
		\frac{24 C_n}{\pi l_nT} + 
		\frac{1}{\pi }\bigg(\int_{-T}^{T} \bigg|\lambda\big(\frac{t}{a_n}\big)^n \frac{\Pi_{\frac{t}{a_n}}f(x) - \nu(f)}{l_nt}\bigg|dt 
		+ \int_{-T}^T \bigg|\frac{(e^{-itb_n}R_{\frac{t}{a_n}}^n-h_\alpha(t)R_0^n)f(x)}{l_nt}\bigg|dt 
		\notag \\
		& \, + \nu(f)\int_{-T}^T \bigg|\frac{\lambda(\frac{t}{a_n})^n - \phi_Z(\frac{t}{a_n})^n}{l_nt}\bigg|dt 		 + \nu(f) \int_{-T}^T \bigg|\frac{e^{-itb_n}\phi_Z(\frac{t}{a_n})^n-h_\alpha(t)+l_nh_\alpha(t)J(t)}{l_nt}\bigg|dt\bigg). 
	\end{align}
	Note that $\rho > \alpha - 2 > -\rho - 2$, which implies that $\rho > -1$. Since $|A|$ is $\rho$-regularly varying (see \cite[Theorem B.2.1 and Remark B.3.15]{haan06extreme}), we know that $l_n=A(a_n)$ satisfies $|l_nT| = |\tau A(a_n)a_n| \to +\infty$ as $n \to \infty$. 
	As $(C_n)$ is bounded, the first term in \eqref{l-n--1-sup-s-in-R-2} tends to $0$. Hence, in order to prove~\eqref{lim-l-n-sup=x}, 
	we only need to show that the four integrals in~\eqref{l-n--1-sup-s-in-R-2}, denoted 
	successively 
	by $I_1, \cdots, I_4$,  tend to $0$ as $n \to \infty$. 
	
	By Part 1 of Proposition~\eqref{prop::useful}, there exists $C > 0$ such that for $n \geq 1$, 
	$$|\Pi_{\frac{t}{a_n}}f(x) - \nu(f)| = O(\| P_{\frac{t}{a_n}} - P\|_\mathcal{L}) \leq  C\big(\frac{|t|}{a_n}\big)^\beta, \quad \forall t \in [-T, T].$$
	Using this inequality and~\eqref{lambda-t-a-n-n-K-t}, we have that for $n \geq N$,
	\begin{equation}
		I_1 \leq \int_{-T}^T K(t) \bigg| \frac{C(\frac{|t|}{a_n})^\beta}{l_nt}\bigg|dt \leq \frac{C}{l_na_n^\beta} \int_\mathbb{R} |t|^{\beta - 1} K(t)dt.
	\end{equation}
	Since $\beta > -\rho$, we have $l_na_n^\beta \to \infty$ as $n \to \infty$, thus $I_1 \to 0$. 
	
	Let $\kappa \in (0, 1)$ be as in Proposition~\ref{prop::useful}. By the decomposition
	\begin{equation}
		e^{-itb_n}R_{\frac{t}{a_n}}^n - h_\alpha(t)R_0^n = e^{-itb_n}(R_{\frac{t}{a_n}}^n-R_0^n) + (1-h_\alpha(t))R_0^n + (e^{-itb_n}-1)R_0^n 
	\end{equation}
	and \eqref{R-t-n-R-0-n}, there exists $C > 0$ such that for $n \geq N$,  
	\begin{equation}\label{I-2-C-n-l-n-1}
		I_2 \leq C\kappa^n l_n^{-1}\int_{|t| \leq T}\bigg| \frac{(\frac{|t|}{a_n})^\beta}{t}\bigg|dt + C\kappa^nl_n^{-1}\int_{|t|\leq T} \bigg|\frac{1-h_\alpha(t)}{t}\bigg|dt + l_n^{-1}b_nO(T\kappa^n).
	\end{equation}
	This implies that $I_2 = o(1) + l_n^{-1}b_nO(T\kappa^n)$ as $n \to \infty$. Using~\cite[Proposition 1, 2]{de1999exact} 
	and~\cite[Page 86, Lemma 2.6.1]{Ibragimov-Linnik-1971}, we know that $b_n=O(1)$. Thus $I_2 \to 0$ as $n \to \infty$. 
	
		By \eqref{lambda-t-a-n-n-phi-Z}, we have that for $n \geq N$, 
	\begin{equation}
		I_3 \leq \int_{\mathbb{R}} \frac{C_1K(t)|t|^{2\beta}}{l_n n^\gamma |t|} dt
	\end{equation}
	Since $\gamma > \frac{-\rho}{\alpha}$, we have $l_nn^\gamma \to +\infty$. Note that $K(t)$ decays faster than any polynomial  of $|t|$ as $t \to \pm\infty$. It follows that $I_3 \to 0$ as $n \to \infty$. 
	
	Set $\eta = \frac{1}{2(\alpha-\rho)}$. As in Lemma \ref{lem:esseen-estimate},  we write $ m_n = |l_n|^{-\eta}$. By~\eqref{lim-n-1-l-n} and integrations on two regions $|t| < m_n$ and $m_n \leq |t| \leq T$, we   get that for $n \geq N$, 
	$$I_4 \leq o(1) + \int_{|t|\geq m_n} \frac{C(K(t) + |h_\alpha(t)|(1+|l_nJ(t)|)}{|l_n|m_n}dt.$$ Since $K(t)$ and $|h_\alpha(t)|$ decay faster than any polynomial 
	of $|t|$ 
	as $t \to \pm\infty$, we get that $I_4 \to 0$ as $n \to \infty$.
	
	It follows that the left hand side of~\eqref{l-n--1-sup-s-in-R-2} tends to $0$ as $n \to \infty$, uniformly in $x \in \mathbb{S}_+^{d-1}$. 
		This shows~\eqref{lim-l-n-sup=x} and proves the theorem for the case $\rho > -\alpha$. 
\end{proof}

\subsection{Proof of Theorem~\ref{thm::main} for the case $\rho < -\alpha$}


For this case, we first establish three lemmas.

	The first lemma introduces the operator $\Delta $ used later in the proof of Theorem~\ref{thm::main}. Recall that  the operator $Q$ is defined in~\eqref{def-Q}, $\Pi$ is the  projection operator 
	(see \eqref{def-Pi}), and $R_0= P-\Pi$ (see  Part 2 of Proposition \ref{prop::useful}). 
	
	\begin{lemma}\label{lem::operator-U}
		Assume Conditions~\ref{cond::Furstenberg-Kesten}, \ref{cond::stable_1} and \ref{cond::when_modulus_large}. Then, the following limit exists in $\mathcal{B}(\mathcal{L})$ equipped with the operator norm $\| \cdot \|_{\mathcal L}$:
		\begin{equation}\label{u-lim-sum}
			\Delta  := \lim_{m \to \infty}\sum_{i=0}^{m-1}P^{m-1-i}(Q-P)P^{i} 
			= \Pi (Q-P)\sum_{i\geq 0} R_0^i. 
		\end{equation} 
		Moreover, we have $\Delta f = \delta (f)\mathbf{1}$ for all $f \in \mathcal{L}$, where $\delta : \mathcal{L} \to \mathbb{C}$ is the bounded linear mapping defined by 
		\begin{equation}\label{def-lambda-f}
			\delta (f) = \nu\Big((Q-P)\sum_{i\geq 0} R_0^if\Big), \quad \forall f \in \mathcal{L}. 
		\end{equation}
	\end{lemma} 
	\begin{proof}
		Since $(Q-P)\Pi = 0$, we get from~\eqref{R-t-n-R-0-n} that for $m \geq 1$, 
		\begin{align}
			\sum_{i=0}^{m-1}P^{m-1-i}(Q-P)P^{i} & = \sum_{i=0}^{m-1}(\Pi + R_0^{m-1-i})(Q-P)R_0^{i} \notag \\
			& = \sum_{i=0}^{m-1}\Pi(Q-P)R_0^i +  \sum_{i=0}^{m-1} R_0^{m-1-i}(Q-P)R_0^i. \label{U}
		\end{align}
			Since $\|R_0^n\|_\mathcal{L} \leq C\kappa^n \;  \forall n \geq 0$,  by Proposition~\ref{prop::useful}, 
		we see that 
		$$		\bigg\|\sum_{i=0}^{m-1} R_0^{m-1-i}(Q-P)R_0^i\bigg\|_\mathcal{L} 
		\leq C^2 m  \kappa^m   \| Q-P \|_{\mathcal L} 
		\to 0,  \quad \mbox{  as  }  m \to \infty. $$
		Hence the limit in~\eqref{u-lim-sum} converges in $\mathcal{B}(\mathcal{L})$ with
		\begin{equation}\label{U-22}
			\Delta  = \sum_{i=0}^\infty \Pi(Q-P)R_0^i = \Pi(Q-P)\sum_{i=0}^\infty R_0^i. 
		\end{equation}
		This implies that for any $f \in \mathcal{L}$, 
		\begin{equation}
			\Delta f = \Pi\Big((Q-P) \sum_{i \geq 0} R_0^i f\Big) = \nu\Big((Q-P)\sum_{i \geq 0} R_0^i f\Big) \mathbf{1} = \delta (f) \mathbf{1}.
		\end{equation}
		This ends the proof of the lemma.  
	\end{proof}
	

	The second lemma concerns the asymptotics of operators $P_t - P$ and $\Pi_t - \Pi$ as $t \to 0$. 
Recall the number $p$ introduced in Condition~\ref{cond::stable_1} and the constant  \mbox{$c = \lim_{s \to +\infty} s^\alpha(1-F(s)+F(-s))$}
defined in  \eqref{c}. For $t \in \mathbb{R}$, define
\begin{equation}
	C(t, \alpha) = -cd_\alpha + i\,\mathrm{sgn}(t)c(2p-1)\alpha d_{\alpha + 1}, 
\end{equation} 
	where $d_a$ is defined in~\eqref{d-alpha} for $a \in (0, 2)$. 
\begin{lemma} \label{lem-311}
	Assume Conditions~\ref{cond::Furstenberg-Kesten}, \ref{cond::stable_1}, \ref{cond::when_modulus_large}, and $\rho < -\alpha$. Then, for any $f \in \mathcal{L}$, 
	as $t \to 0$, 
	\begin{align}
		& \|(P_t - P)f - C(t, \alpha)|t|^\alpha Qf\|_\infty = o(|t|^\alpha),
		\label{P-t--P-f-C-t-alpha} \\
		& \|(\Pi_t - \Pi)f - C(t, \alpha)|t|^\alpha \Delta f\|_\infty = o(|t|^\alpha).
		\label{Pi-t-Pi-f-C-t-alpha-t-alpha}
	\end{align}
\end{lemma}
\begin{proof}
	(1)  We first prove ~\eqref{P-t--P-f-C-t-alpha}. 
	Let $f \in \mathcal{L}$.
	By Lemma~\ref{lem::bound},  we know that 
	\begin{equation}
		\sup_{x \in \mathbb{S}_+^{d-1}, g \in \mathrm{supp}(\mu)}|\log\|g\| - \sigma(g, x)| < +\infty,
	\end{equation}
	so that  uniformly for $x \in \mathbb{S}_+^{d-1}$, 
	\begin{equation}\label{P-t-P}
		(P_t-P)f(x)=\int_{\mathbb{R}} (e^{it\log \|g\|}-1)f(g\cdot x)d\mu(g) + O(|t|), \quad \mbox{ as } t \to 0,
	\end{equation}
	where  $O(|t|)$ means a real number (depending on $x$ and $t$)  bounded by $C |t|$,  for some  constant $C$ independent  of $x$.  
	
	Now,  for $x \in \mathbb{S}_+^{d-1}$, we have,  with  $G_1(s) = e^{its}-1$ and  $G_2(s) = \int f(g\cdot x) \mathbbm{1}_{\{ \log \|g\| \leq s \}} d\mu(g)$, 
	\begin{align}
		& \int_{\mathbb{R}} (e^{it\log \|g\|}-1)f(g\cdot x)d\mu(g) = \int_\mathbb{R} G_1(s) dG_2(s).
	\end{align}
	We come to estimate the integral  $ \int_0^\infty G_1(s) dG_2(s)$.
	Using integration by parts, we get
	\begin{equation}\label{int-0-infty-G-1-d-G-2}
		\int_0^{+\infty} G_1(s)dG_2(s) = \int_{0}^{+\infty} G_1(s)d(G_2(s)-G_2(+\infty)) = \int_{0}^{+\infty} (G_2(+\infty) - G_2(s))ite^{its}ds,
	\end{equation}
	where $G_2(+\infty)=\int f(g\cdot x) d\mu(g)$. Note that
	\begin{equation}\label{G-2}
		G_2(+\infty)-G_2(s) = \mathbb{E}[f(A_1 \cdot x) | \log \| A_1 \| > s] \cdot \mathbb{P}[\log \|A_1\| > s].
	\end{equation}
	Let $D$ be the set of nonnegative matrices $g$ with operator norm one that satisfy the Furstenberg-Kesten condition ~\eqref{0-leq-max},  
	and $D' = \{ g\cdot y : g \in D, y \in \mathbb{S}_+^{d-1} \}$. 
	Note that there exists a constant $C_1 > 0$ such that $\mathbf{d}(y_1, y_2) \leq C_1 |y_1 - y_2| $ for $y_1, y_2 \in D'$ (we can show this by using definition of $\mathbf{d}$ and noticing that entries of $y_i$ are bounded uniformly from zero). 
		By \cite[Lemma 5.1] {mei2024support}, there exists $C_2 > 0$ such that $|gy| \geq C_2 \|g\| = C_2$ for all $g \in D$ and $y \in \mathbb{S}_+^{d-1}$. Therefore, 
		$$\mathbf{d}(g_1\cdot y, g_2 \cdot y) \leq C_1 \bigg|\frac{g_1y}{|g_1y|} - \frac{g_2y}{|g_2y|}\bigg| \leq 2C_1\frac{|(g_2 - g_1)y|}{|g_1y|} \leq  \frac{2C_1}{C_2} \|g_2 - g_1\|,  \quad   \forall g_1, g_2 \in D, y \in \mathbb{S}_+^{d-1}.$$
	Using this inequality and Condition~\ref{cond::when_modulus_large},  we have $\mathbb{E}[f(A_1 \cdot x) | \log \| A_1 \| > s] \to Qf(x)$ as $s \to +\infty$, uniformly in $x$. 
	Since $1 - F(s) = cps^{-\alpha} (1 + o(1))$, 
	from the first inequality of~\eqref{sigma-A-1-x} we deduce that $\mathbb{P}[\log \|A_1\| > s] = cps^{-\alpha}(1+o(1))$. 
	Thus from~\eqref{G-2}, we get that, uniformly in $x \in \mathbb{S}_+^{d-1}$, 
	\begin{equation}
		\lim_{s \to +\infty} s^{\alpha} (G_2(+\infty) - G_2(s)) = cp Qf(x).
	\end{equation}
	Write $h(s) = s^\alpha(G_2(+\infty) - G_2(s))$. For $t > 0$,
	\begin{equation}
		\int_0^{+\infty} (G_2(+\infty) - G_2(s))ite^{its}ds = \int_0^{+\infty} ith(s)e^{its}s^{-\alpha}ds = it^\alpha \int_0^{+\infty} \frac{h(\frac{s}{t})}{s^\alpha}e^{is}ds. 
	\end{equation}
	Since $G_2(+\infty)-G_2(s)$ is decreasing, from 
		\cite[Page 86, Lemma 2.6.1]{Ibragimov-Linnik-1971},
	this implies  that for $t  > 0$, 
	\begin{align}
		\int_0^{+\infty} (G_2(+\infty) - G_2(s))ite^{its}ds & = it^{\alpha}\Big(h\big(\frac{1}{s}\big) + o(1)\Big)\int_0^{+\infty}\frac{e^{is}}{s^\alpha}ds \notag \\
		&  = it^\alpha (cp\,Qf(x) + o(1)) (\alpha d_{\alpha + 1} + id_\alpha), 
	\end{align}
	where we use the fact that $\int_0^{+\infty} s^{-\alpha} \cos s \, ds = \alpha \int_0^{+\infty} s^{-\alpha-1}\sin s \, ds = \alpha d_{\alpha + 1}$ (by integration by parts). Similarly, we can estimate~\eqref{int-0-infty-G-1-d-G-2} for the case $t < 0$. The same argument applies for estimating $\int_{-\infty}^{0} G_1(s)dG_2(s)$ for $t \in \mathbb R$. This leads to 
	\begin{align} \label{o-inf-t-a}
		(P_t - P)f = C(t, \alpha)|t|^\alpha Qf + \varepsilon_t  |t|^\alpha + C_t |t|  \quad  \mbox{ as } t\to 0, 
	\end{align} 
	where  $ \varepsilon_t  $ and   $C_t$  are functions   on $\mathbb S_+^{d-1}$ 
	satisfying   $ \lim_{t \to 0} \| \varepsilon_t \|_\infty =0 $ and  $ \sup_{ 0 < |t | \leq t_0} \|  C_t \|_\infty < \infty $ for some $ t_0 >0 $ small enough.   
	Since $\alpha-2<\rho<-\alpha$, we have $\alpha < 1$, hence $|t| = o(|t|^\alpha)$. It follows that~\eqref{P-t--P-f-C-t-alpha} holds. 
	
	(2)  We then prove  \eqref{Pi-t-Pi-f-C-t-alpha-t-alpha}. 
	Let $m \geq 1$. Recall that from Lemma~\ref{lem::bound}, it holds that $\|P_t-P\|_\mathcal{L} = O(|t|^\alpha)$  as $t \to 0$. Since $\phi_Z(t) = \nu(P_t\mathbf{1})$, we deduce from~\eqref{P-t--P-f-C-t-alpha} (with $f = \mathbf{1}$) 
	and Part 3 of Proposition~\ref{prop::useful} that 
	\begin{equation}\label{1-lambda-t-C-t-alpha}
		1 - \lambda(t) = -C(t, \alpha)|t|^\alpha + o(|t|^\alpha).
	\end{equation}
	
	We come to expand $P_t^mf - \lambda(t)^mP_0f$ in two ways. On the one hand, using \eqref{R-t-n-R-0-n}, \eqref{P-t-P-L-O} and~\eqref{1-lambda-t-C-t-alpha}, we notice that as $t \to 0$,
	\begin{align}\label{p-t-m-f}
		P_t^mf - \lambda(t)^mP_0f & = \lambda(t)^m(\Pi_t-\Pi)f+\big((R_t^m-R_0^m)f + (1 - \lambda(t)^m)R_0^mf\big), \notag \\
		& = (\Pi_t - \Pi)f + (\lambda(t)^m-1)(\Pi_t - \Pi)f 
		+  C^{(1)}_{m, t} m\kappa^m 
		\notag \\
		& = (\Pi_t - \Pi)f +  C^{(2)}_{m, t} m  |t|^{2\alpha} 
		+   C^{(1)}_{m, t}    m\kappa^m  |t|^\alpha,
	\end{align}
	where for $i=1,2$, 
	$C^{(i)}_{m, t}$  are functions  on $\mathbb{S}_+^{d-1}$
	satisfying  $ \sup_{m \geq 1} \sup_{ 0<| t | \leq t_i } \| C^{(i)}_{m, t} \|_\infty  < \infty$, for some $t_i >0$  small enough. 
	On the other hand, using~\eqref{P-t--P-f-C-t-alpha} and \eqref{1-lambda-t-C-t-alpha}, we have
	\begin{align}\label{ptmf}
		P_t^mf - \lambda(t)^mP_0f & = (P_t^m-P_0^m) f + (1-\lambda(t)^m)P_0^mf \notag \\
		& =  \sum_{i=0}^{m-1} P_t^{m-1-i}(P_t-P_0)P_0^{i} f + (1-\lambda(t)^m)P_0^mf.\notag \\
		& = C(t, \alpha)|t|^\alpha \sum_{i=0}^{m-1}P_t^{m-1-i}QP_0^{i} f- mC(t, \alpha)|t|^\alpha P_0^mf +  \varepsilon_{m, t} m |t|^\alpha. \notag \\
		& = C(t, \alpha)|t|^\alpha \sum_{i=0}^{m-1}P^{m-1-i}(Q-P)P^{i}f + 
		\varepsilon_{m, t} m  |t|^\alpha,
	\end{align}
	where $  \varepsilon_{m, t}  $ are functions on $\mathbb{S}_+^{d-1}$ satisfying  $\lim_{t \to 0}  \sup_{m \geq 1}  \| \varepsilon_{m, t}  \|_\infty  = 0$.
	Comparing the above two expansions of $P_t^mf - \lambda(t)^mP_0f$, we get that for some $C > 0$, as $t\to 0$,  
	\begin{equation}\label{Pi-t-Pi-f-C-t-alpha}
		\bigg\|(\Pi_t - \Pi)f - C(t, \alpha)|t|^\alpha \sum_{i=0}^{m-1}P^{m-1-i}(Q-P)P^{i}\bigg\|_\infty \leq C(m\kappa^m|t|^\alpha + m|t|^{2\alpha}) + m \,o(|t|^\alpha).
	\end{equation} 
	Using~\eqref{U}, \eqref{U-22} and the property that $\|R_0^n\|_\mathcal{L} = O(\kappa^n) $ as $n \to \infty$ (see Proposition~\ref{prop::useful}), we know  that there exists $C' > 0$ such that 
	\begin{equation}\label{sum-i-0-m-1-P-m-1-i}
		\bigg\|\sum_{i=0}^{m-1}P^{m-1-i}(Q-P)P^{i} - \Delta \bigg\|_\mathcal{L} \leq C'm\kappa^m, \quad \forall m \geq 1. 
	\end{equation}
	It follows from~\eqref{Pi-t-Pi-f-C-t-alpha} and~\eqref{sum-i-0-m-1-P-m-1-i} that 
	\begin{equation}
		\| (\Pi_t - \Pi)f - C(t, \alpha)|t|^\alpha \Delta f\|_\infty \leq (C + C')(m\kappa^m|t|^\alpha + m|t|^{2\alpha}) + m\, o(|t|^\alpha). 
	\end{equation} 
	Passing to the limit as $t \to 0$ and then as $m \to \infty$, we get ~\eqref{Pi-t-Pi-f-C-t-alpha-t-alpha}. 
\end{proof}

The third lemma is a technical result that improves the estimation~\eqref{lambda-t-a-n-n-phi-Z} in Lemma~\ref{lem::decay_1}.
\begin{lemma}\label{lem::bootstrap}
	Assume Conditions~\ref{cond::Furstenberg-Kesten}, \ref{cond::stable_1}, \ref{cond::when_modulus_large}, and $\rho < -\alpha$. Then, for any $\varepsilon >0$, there exist positive numbers $N_0, \tau, C$ such that for all $n \geq N_0$ and $ t \in [-\tau a_n, \tau a_n]$,
	\begin{equation}\label{lambda-t-a-n-n-phi-Z-n-t-a-n-leq-C}
		\Big|\lambda\big(\frac{t}{a_n}\big)^n - \phi_Z\big(\frac{t}{a_n}\big)^n\Big| \leq \varepsilon
		K(t)
		|t|^{2\alpha}n^{-1},
	\end{equation}
	where $K(t) := \exp(-C |t|^{\alpha}\min(|t|^{\frac{\alpha}{2}}, |t|^{-\frac{\alpha}{2}}))$. 
\end{lemma}
\begin{proof}
	
	For $t \in I$, let $v_t = \frac{\Pi_t\mathbf{1}}{\nu(\Pi_t\mathbf{1})}$, 
	which is an eigenfunction of $P_t$: $P_tv_t=\lambda(t)v_t$. Note that $\Pi_t^2 = \Pi_t$ and
	\begin{equation}\label{v-t-1-pi-t-2}
		v_t - \mathbf{1} = \frac{(\Pi_t^2 -\Pi\Pi_t)\mathbf{1}}{\nu(\Pi_t\mathbf{1})} = \frac{(\Pi_t - \Pi)^2\mathbf{1} + (\Pi_t - \Pi)\mathbf{1}}{\nu(\Pi_t\mathbf{1})}.
	\end{equation}
	Since $\lambda(t) = \nu(P_t v_t)$, $\phi_Z(t) = \nu(P_t\mathbf{1})$ and $(\nu P) (v_t-\mathbf{1}) = \nu(v_t-\mathbf{1}) = 0$,  from \eqref{v-t-1-pi-t-2}
	we get that as $t \to 0$,
	\begin{equation}\label{lambda-t-a-n-phi-Z-t-a-n}
		\lambda(t) - \phi_Z(t) = \nu(P_{t}(v_t - \mathbf{1})) = \nu((P_{t}-P)(v_t - \mathbf{1})) = O(\|(P_{t} - P)(\Pi_{t} - \Pi)\mathbf{1}\|_\infty).
	\end{equation}
	From~\eqref{P-t-P-L-O}, \eqref{v-t-1-pi-t-2} and \eqref{lambda-t-a-n-phi-Z-t-a-n}, we deduce  that as $t \to 0$, 
	\begin{equation}
		\lambda(t) - \phi_Z(t)  
		= |t|^\alpha O(\|(\Pi_{t} - \Pi)\mathbf{1}\|_\infty).
	\end{equation} 
	Since  $\Delta \mathbf{1}=0$, we get from~\eqref{P-t--P-f-C-t-alpha} that $\|(\Pi_{t} - \Pi)\mathbf{1}\|_\infty = o(|t|^\alpha )$, hence as $t \to 0$,
	\begin{equation}\label{}
		\lambda(t) - \phi_Z(t) = o(|t|^{2\alpha}).
	\end{equation}
	This implies that for any $\varepsilon > 0$, there exists $\tau > 0$ such that 
	\begin{equation}
		\Big|\lambda\big(\frac{t}{a_n}\big) - \phi_Z\big(\frac{t}{a_n}\big)\Big| \leq  \varepsilon|t|^{2\alpha}n^{-2}, \quad \forall t \in [-\tau a_n, \tau a_n]
	\end{equation}
	(notice that $a_n = n^{\frac{1}{\alpha}}$ since $\rho < -\alpha < 0$). It follows from \eqref{lambda-t-a-n-n-K-t} and \eqref{lambda-t-a-22} that~\eqref{lambda-t-a-n-n-phi-Z-n-t-a-n-leq-C} holds. 
\end{proof}

Now we come to finish the proof of Theorem~\ref{thm::main}.
\begin{proof}[Proof of Theorem~\ref{thm::main}, case $\rho < -\alpha$]
	Let $f \in \mathcal{L}$, and assume that $f$ is real-valued and $\min_{y \in \mathbb{S}_+^{d-1}}f(y) > 0$ 
	without loss of generality. 
		Note that $\log h_\alpha(t) = C(t, \alpha)|t|^\alpha$ for $t \in \mathbb{R}$, and recall the functions $M$ and $N$ defined in \eqref{def-MsNs}:
		\begin{align}
			M(s) & = \frac{1}{2\pi} \int_\mathbb{R} \frac{e^{-its}}{it}J(t)h_\alpha(t)dt,
			\quad \\
			N(s)& =\frac{1}{2\pi}\int_\mathbb{R}\frac{e^{-its}}{-it}(\log h_\alpha(t)) h_\alpha(t)dt = \frac{1}{2\pi}\int_\mathbb{R}\frac{e^{-its}}{-it}C(t, \alpha)|t|^\alpha h_\alpha(t)dt. 
		\end{align}
	
	To prove \eqref{lim-n-bigg-mathbb-E}, we need to show that as $n \to \infty$, 
	\begin{equation}
		\sup_{s \in \mathbb{R}, x \in \mathbb{S}_+^{d-1}}n\bigg|\mathbb{E}[f(X_n^x) \mathbbm{1}_{\{\frac{S_n^x}{a_n}-b_n \leq s\}}] - \nu(f) H_\alpha(s) - n^{-1}\nu(f)M(s)-n^{-1}\Delta f(x)N(s)\bigg| = 0.
	\end{equation}
	Since  $ \sup_{x \in \mathbb{S}_+^{d-1}}  | \mathbb{E}[f(X_n^x)] - \nu (f)  | 
	=  \sup_{x \in \mathbb{S}_+^{d-1}} | (\Pi + R_0^n)f(x) - \nu(f) |  = \|R_0^n\|_\infty \to 0$ exponentially fast as $n \to \infty$, it suffices to show 
	\begin{equation}\label{lim-s-x-E-f}
		 \sup_{s \in \mathbb{R}, x \in \mathbb{S}_+^{d-1}} n\bigg|\mathbb{E}[f(X_n^x) \mathbbm{1}_{\{\frac{S_n^x}{a_n}-b_n \leq s\}}] - \mathbb{E}[f(X_n^x)]H_\alpha(s) - n^{-1}\nu(f)M(s)-n^{-1}\Delta f(x)N(s)\bigg| \to 0. 
	\end{equation}
	%
	
		By Lemma~\ref{lem::decay_1}, there exist positive numbers $\tau, N, C$ such that \eqref{lambda-t-a-n-n-K-t} hold for all $t \in [-\tau a_n, \tau a_n]$ and $n \geq N$, with $K(t) = \exp(-C |t|^{\alpha}\min(|t|^{\frac{\alpha}{2}}, |t|^{-\frac{\alpha}{2}}))$. Set $T={\tau}{a_n}={\tau}{n^{\frac{1}{\alpha}}}$. 
	Using Lemma~\ref{lem::Esseen} with $F_1(s) = \mathbb{E}[f(X_n^x)\mathbbm{1}_{\{\frac{S_n^x}{a_n}-b_n \leq s\}}]$, $F_2(s) = \mathbb{E}[f(X_n^x)]H_\alpha(s)$ and $G(s) = n^{-1}(\nu(f)M(s) + \Delta f(x)N(s))$, we get,
	for any $x \in \mathbb{S}_+^{d-1}$ and $n \geq 1$, with $C_n = \mathbb{E}[f(X_n^x)] \|H_\alpha'\|_\infty + n^{-1}(\nu(f)\|M'\|_\infty + |\Delta f(x)|\|N'\|_\infty)$, 
	\begin{align}\label{l-n-1-sup-s-in-R-E-2}
		&  \sup_{s \in \mathbb{R}}n\bigg|\mathbb{E}[f(X_n^x) \mathbbm{1}_{\{\frac{S_n^x}{a_n}-b_n \leq s\}}] - \mathbb{E}[f(X_n^x)]H_\alpha(s) - n^{-1}\nu(f)M(s)-n^{-1}\Delta f(x)N(s)\bigg| \notag \\
		& \leq \frac{24 C_n}{\pi n^{-1}T} \notag \\
		& + \int_{-T}^T \bigg| \frac{e^{-itb_n}P_{\frac{t}{a_n}}^nf(x) - h_\alpha(t)P_0^nf(x) + n^{-1} \nu(f)h_\alpha(t)J(t) - n^{-1}C(t, \alpha)|t|^\alpha h_\alpha(t) \Delta f(x)}{\pi n^{-1}t} \bigg|dt.
	\end{align}
	From \eqref{R-t-n-R-0-n}, we have, for $t \in [-T, T]$ and $x \in \mathbb{S}^{d-1}_+$, 
	\begin{align}
		& e^{-itb_n}P_{\frac{t}{a_n}}^nf(x) - h_\alpha(t)P_0^nf(x) + n^{-1} \nu(f)h_\alpha(t)J(t) - n^{-1}C(t, \alpha)|t|^\alpha h_\alpha(t) \Delta f(x) \notag \\
		= & e^{-itb_n}\lambda\big(\frac{t}{a_n}\big)^n\big(\Pi_{\frac{t}{a_n}}f-\nu(f)-n^{-1}C(t, \alpha)|t|^\alpha \Delta f(x)\big) + (e^{-itb_n}R_{\frac{t}{a_n}}-h_\alpha(t)R_0^n)f(x) \notag \\
		& \quad + \nu(f) e^{-itb_n} \Big(\lambda\big(\frac{t}{a_n}\big)^n - \phi_Z\big(\frac{t}{a_n}\big)^n\Big) + \nu(f)\Big(e^{-itb_n}\phi_Z(\frac{t}{a_n})^n-h_\alpha(t)+n^{-1}h_\alpha(t)J(t)\Big) \notag \\
		&\quad  +e^{-itb_n}\Big(\lambda\big(\frac{t}{a_n}\big)^n - \phi_Z\big(\frac{t}{a_n}\big)^n\Big)n^{-1}C(t, \alpha)|t|^\alpha \Delta f(x) \notag \\
		& \quad + \Big(e^{-itb_n}\phi_Z(\frac{t}{a_n})^n - h_\alpha(t)\Big)n^{-1}C(t, \alpha)|t|^\alpha \Delta f(x). 
	\end{align}
	Plugging this into~\eqref{l-n-1-sup-s-in-R-E-2}, we get 
	\begin{align}
		 & \sup_{s \in \mathbb{R}}n\bigg|\mathbb{E}[f(X_n^x) \mathbbm{1}_{\{\frac{S_n^x}{a_n}-b_n \leq s\}}] - \mathbb{E}[f(X_n^x)]H_\alpha(s) - n^{-1}\nu(f)M(s)-n^{-1}\Delta f(x)N(s)\bigg| \notag \\
		 & \leq \frac{24 C_n}{\pi n^{-1}T} + \frac{1}{\pi}\sum_{i=1}^6 I_i',
	\end{align}
	where
	\begin{align}
		I_1' & = \int_{-T}^T \big|\lambda(\frac{t}{a_n})\big|^n\bigg|\frac{\Pi_{\frac{t}{a_n}}f(x)-\nu(f)-n^{-1}C(t, \alpha)|t|^\alpha \Delta f(x)}{n^{-1}t}\bigg|dt, \notag \\
		I_2' & = \int_{-T}^T \bigg|\frac{(e^{-itb_n}R_{\frac{t}{a_n}}^n-h_\alpha(t)R_0^n)f(x)}{n^{-1}t}\bigg|dt,  \notag \\
		I_3' & = \nu(f)\int_{-T}^T \bigg|\frac{\lambda(\frac{t}{a_n})^n - \phi_Z(\frac{t}{a_n})^n}{n^{-1}t}\bigg|dt, \notag \\
		I_4' & = \nu(f) \int_{-T}^T \bigg|\frac{e^{-itb_n}\phi_Z(\frac{t}{a_n})^n-h_\alpha(t)+n^{-1}h_\alpha(t)J(t)}{n^{-1}t}\bigg|dt, \notag \\
		I_5' & = \int_{-T}^T \bigg|\frac{ (\lambda(\frac{t}{a_n})^n - \phi_Z(\frac{t}{a_n})^n)n^{-1}C(t, \alpha)|t|^\alpha \Delta f(x) }{n^{-1}t} \bigg|dt, \notag \\
		I_6' & = \int_{-T}^T \bigg| \frac{ (e^{-itb_n}\phi_Z(\frac{t}{a_n})^n - h_\alpha(t))n^{-1}C(t, \alpha)|t|^\alpha \Delta f(x) }{n^{-1} t} \bigg|dt. \notag
	\end{align}
	Since $(C_n)$ is bounded and $n^{-1}T = \tau n^{-1+\frac{1}{\alpha}} \to \infty$ as $n \to \infty$, we see that in order to prove~\eqref{lim-s-x-E-f}, it suffices to prove that $I_1', \cdots, I_6'$ tend to $0$.
	
		For any $\varepsilon > 0$, by \eqref{lambda-t-a-n-n-K-t} and \eqref{Pi-t-Pi-f-C-t-alpha-t-alpha}, there exist $\xi > 0$ and $C > 0$ such that for all $n \geq 1$, 
		\begin{equation}
			I_1' \leq \int_{|t| \leq \xi a_n}  K(t)\varepsilon dt + \int_{\xi a_n \leq |t| \leq T} K(t) \bigg|\frac{C}{n^{-1}t}\bigg|dt \leq \varepsilon \int_\mathbb{R} K(t) dt + Cn \int_{|t| \geq \xi a_n} \frac{K(t)}{|t|}dt. 
		\end{equation}
		Since $K(t)$ decays faster than any polynomial  of $|t|$ as $t \to \pm\infty$, we know that 
		\begin{equation}
			\limsup_{n \to \infty} I_1' \leq \varepsilon \int_{\mathbb{R}}K(t)dt.
		\end{equation}
		Taking $\epsilon \to 0+$, we get that $I_1' \to 0$ as $n \to \infty$. 
		
		We see that $I_2' \to 0$ by arguing as in the proof of the case $\rho > -\alpha$ (see~\eqref{I-2-C-n-l-n-1}). 
		
		For any $\varepsilon > 0$, by \eqref{lambda-t-a-n-n-K-t} and \eqref{lambda-t-a-n-n-phi-Z-n-t-a-n-leq-C}, there exist $\xi' > 0$  such that for $n \geq 1$, 
		\begin{align}\label{I-3-prime}
			I_3' &  \leq \nu(f) \int_{|t| \leq \xi' a_n}  \frac{\varepsilon K(t) |t|^{2\alpha} n^{-1}}{n^{-1}|t|} dt + \nu(f) \int_{\xi' a_n \leq |t| \leq T} \frac{2K(t)}{n^{-1}|t|}dt \notag \\
			& \leq \nu(f) \int_\mathbb{R} \varepsilon K(t) |t|^{2\alpha-1}dt + \nu(f) 2n \int_{|t| \geq \xi' a_n} \frac{K(t)}{|t|}dt.
		\end{align}
		Passing to the limit as $n \to \infty$ and then as $\varepsilon \to 0+$, we get that $I_3' \to 0$ as $n \to \infty$. 
		
		We then see that $I_4' \to 0$ by using \eqref{lim-n-t-leq-m-n} of Lemma~\ref{lem:esseen-estimate} and arguing as in the proof of the case $\rho > -\alpha$, $I_5' \to 0$ by arguing as in the proof for $I_3' \to 0$ above (see~\eqref{I-3-prime}), $I_6' \to 0$ by truncating the integral and then using~\eqref{lambda-t-a-n-n-K-t} together with convergence $e^{-itb_n}\phi_Z(\frac{t}{a_n})^n \to h_\alpha(t)$. Thus \eqref{lim-s-x-E-f} holds.  This ends the proof  of Theorem~\ref{thm::main} for the case $\rho < -\alpha$. 

\end{proof}

\bmhead{Acknowledgements}
The work has been supported 
by  
the ANR project ``Rawabranch'' number ANR-23-CE40-0008,  the France 2030 framework program, Centre Henri Lebesgue ANR-11-LABX-0020-01, and the 
Tsinghua Scholarship for Overseas Graduate Studies.  
The work has benefited from Jianzhang Mei’s visit to the University of South Brittany. The hospitality of the university was greatly appreciated.
\bibliography{ref}

\end{document}